\documentclass[letterpaper,12pt]{amsart}

\usepackage[all]{xy}                        %
\usepackage[margin=1in]{geometry}
\input xypic
\usepackage[bookmarks=true]{hyperref}       

\usepackage{amssymb,latexsym,amsmath,amscd}
\usepackage{xspace}
\usepackage{color}
\allowdisplaybreaks
\usepackage{caption}
\usepackage{enumitem}

\usepackage{mathtools}
\usepackage{tikz}
\usetikzlibrary{cd}
\usepackage{graphicx} 

\usepackage[utf8]{inputenc}
\usepackage{fourier}
\usepackage{array}
\usepackage{makecell}
\usepackage{blkarray}

\usepackage{array}
\newcolumntype{L}[1]{>{\raggedright\let\newline\\\arraybackslash\hspace{0pt}}m{#1}}
\newcolumntype{C}[1]{>{\centering\let\newline\\\arraybackslash\hspace{0pt}}m{#1}}
\newcolumntype{R}[1]{>{\raggedleft\let\newline\\\arraybackslash\hspace{0pt}}m{#1}}

\DeclareUnicodeCharacter{02B9}{\ensuremath{'}}

\theoremstyle{plain}
\newtheorem{theorem}{Theorem}[section]
\newtheorem*{theorem*}{Theorem}
\newtheorem{proposition}[theorem]{Proposition}
\newtheorem{corollary}[theorem]{Corollary}

\theoremstyle{definition}
\newtheorem{definition}[theorem]{Definition}

\newtheorem{remark}[theorem]{Remark}
\newtheorem{example}[theorem]{Example}

\newcommand{\enm}[1]{\ensuremath{#1}}          %
\newcommand{\op}[1]{\operatorname{#1}}
\newcommand{\cal}[1]{\mathcal{#1}}

\renewcommand{\bar}[1]{\overline{#1}}

\newcommand{\bigslant}[2]{{\raisebox{.3em}{$#1$}\left/\raisebox{-.3em}{$#2$}\right.}}

\newcommand{\CC}{\enm{\mathbb{C}}}

\newcommand{\NN}{\enm{\mathbb{N}}}

\newcommand{\ZZ}{\enm{\mathbb{Z}}}

\newcommand{\PP}{\enm{\mathbb{P}}}

\newcommand{\VV}{\enm{\mathbb{V}}}

\newcommand{\Aa}{\enm{\cal{A}}}
\newcommand{\Bb}{\enm{\cal{B}}}

\newcommand{\Ee}{\enm{\cal{E}}}
\newcommand{\Ff}{\enm{\cal{F}}}
\newcommand{\Gg}{\enm{\cal{G}}}
\newcommand{\Hh}{\enm{\cal{H}}}
\newcommand{\Ii}{\enm{\cal{I}}}
\newcommand{\Jj}{\enm{\cal{J}}}
\newcommand{\Kk}{\enm{\cal{K}}}
\newcommand{\Ll}{\enm{\cal{L}}}
\newcommand{\Mm}{\enm{\cal{M}}}
\newcommand{\Nn}{\enm{\cal{N}}}
\newcommand{\Oo}{\enm{\cal{O}}}
\newcommand{\Pp}{\enm{\cal{P}}}
\newcommand{\Qq}{\enm{\cal{Q}}}

\newcommand{\Ss}{\enm{\cal{S}}}
\newcommand{\Tt}{\enm{\cal{T}}}

\newcommand*{\barrow}{\xrightarrow{\hspace*{1.1cm}}}
\renewcommand{\phi}{\varphi}
\renewcommand{\theta}{\vartheta}
\renewcommand{\epsilon}{\varepsilon}

\newcommand{\Ext}{\op{Ext}}

\newcommand{\tensor}{\otimes}         
\newcommand{\dirsum}{\oplus}

\newcommand{\Dirsum}{\bigoplus}

\renewcommand{\to}[1][]{\xrightarrow{\ #1\ }}

\newcommand{\ie}{\textit{i.e.},\ }           

\newcommand{\old}[1]{}

\title{Net logarithmic tangent sheaves of complete intersections}
\author{Sukmoon Huh and Min-Gyo Jeong }
\address{Sungkyunkwan University, Suwon 440-746, Korea}
\email{sukmoonh@skku.edu}

\address{Sungkyunkwan University, Suwon 440-746, Korea}
\email{nida08@skku.edu}

\thanks{SH is supported by the National Research Foundation of Korea(NRF) grant funded by the Korea government(MSIT) (No. RS-2023-00208874). MJ is supported by Basic Science Research Program through the National Research Foundation of Korea(NRF) funded by the Ministry of Education (No. RS-2019-NR040081)}

\begin{document}

\begin{abstract}
The main purpose of this paper is to define the {\it net logarithmic tangent sheaf}, as a generalization of the logarithmic tangent sheaf introduced by P.~Deligne, over the field of complex numbers, and prove some basic properties and give some applications. The generalization is valid for the pairs of the smooth complete intersection variety and its complete intersecting subvariety. As applications, we investigate the locus of net logarithmic tangent sheaves on a smooth cubic surface in the corresponding moduli space of semistable sheaves. 
\end{abstract}

\maketitle

\section{Introduction}
 Let $X$ be a smooth projective variety and let $D$ be a reduced and effective divisor on $X$ with simple normal crossings. Ever since P.~Deligne defined the logarithmic sheaf $\Omega_X^1(\log D)$ as the sheaf of differential $1$-forms on $X$ with logarithmic poles along $D$ in \cite{De}, there has been a huge amount of works on this object with various aspects, including the Terao conjecture and the Torelli properties; see \cite{OT} and \cite{dolgachev1993arrangements,FMV}. A decade later, the notion was extended by K.~Saito in \cite{Saito} with $D$ allowing any reduced singularities and non-normal crossings.
 
 One of the seminal work about the Torelli property of logarithmic sheaves, is the one by I.~Dolgachev and M.~Kapranov in \cite{dolgachev1993arrangements}, where they prove that a generic point in a certain moduli space of stable sheaves on the projective space is given as the logarithmic vector bundle. More specifically, they showed that a generic arrangement of $m$ hyperplanes which do not osculate a normal rational curve of degree $n$ was determined by the logarithmic sheaf $\Omega^1_{\PP^n}(\log D)$ if $m\geq 2n+3$. In other words, the natural correspondence of $D$ with $\Omega^1_{\PP^n}(\log D)$ defines a map 
 \begin{equation*}
\begin{array}{cccc}
\Phi : &\mathbf{F} &\longrightarrow &\mathbf{M}_{\PP^n}(c_1,\ldots,c_n) \\[.3cm]
 &D &\mapsto &\Omega^1_{\PP^n}(\log D)
\end{array}  
\end{equation*}

\noindent which is now shown to be \emph{generically injective}, where $\mathbf{F}$ is the family of divisors $D=D_1+\cdots+D_m$ and $\mathbf{M}_{\PP^n}(c_1,\ldots,c_n)$ is the moduli space of stable vector bundles of rank $n$ on $\PP^n$ with its Chern classes $(c_1,\ldots,c_n)$ given by Chern polynomial $(1-t)^{n+1-m}$; this kind of problem is called the \emph{the Torelli problem}. This result was later extended to $m\geq n+2$ by J.~Vall\'es in \cite{Va2000}, and generalized to arrangements of higher degree hypersurfaces by E.~Angelini in \cite{An2014log,An2015log}. Not only arrangements of hypersurfaces with simple normal crossings in $\PP^n$, but there are also numerous studies exploring various generalized settings; see \cite{BHM,HJblownup,HMPV}. Quite recently, two main generalizations were proposed in \cite{CHKS,dolgachev2007logarithmic} by generalizing the corresponding Poincar\'e residue map, basically encoding more information on the singularity of $D$ into the generalized logarithmic sheaf. This line of research is important in the sense that it can suggest a method to describe moduli space of stable sheaves on $X$, specially non-generic boundary points that are not realized as the logarithmic sheaves in the original sense. As a further extension of this line of works, the work in progress in \cite{HMP} suggests the notion of the {\it net logarithmic sheaf} $\mathring{\Omega}_X^1(\log D)$ on $X$ associated to any divisor $D$, possibly non-reduced.

In this article, we introduce the dual notion of the net logarithmic sheaf $\mathring{\Omega}_X^1(\log D)$. For two complete intersection subvarieties $X=\VV(F_1,\ldots,F_r)$ and $Y=\VV(G_1,\ldots,G_s)$ in $\PP^N$ such that $X\cap Y$ is also a complete intersection, one can consider the gradient map
\[
\nabla~:~\Oo_{\PP^N}(1)^{\dirsum(N+1)} \longrightarrow \left(\Dirsum_{i=1}^{r}\Oo_{\PP^N}(f_i)\right)\dirsum\left(\Dirsum_{j=1}^{s}\Oo_{\PP^N}(g_j)\right).
\]
Then the {\it net logarithmic tangent sheaf} $\Tt_{X}(Y;\PP^N)$ is defined to be $\ker(\nabla)\otimes \Oo_X$, modulo its torsion. It is a generalization of the logarithmic tangent sheaf $\Tt_{X}(-\log D)\cong\left(\Omega^1_{X}(\log D)\right)^{\vee}$, when $X$ and $Y$ are hypersurfaces with possible singularities in $D=X \cap Y$. Note that the sheaf $\Tt_X(Y;\PP^N)$ is torsion-free on $X$ with rank $\dim(X)-\mathrm{codim}(Y)+1$, not necessarily reflexive. If we consider the case $X=\PP^N$, then $\Tt_{\PP^N} (Y; \PP^N)\tensor\Oo_{\PP^N}(-1)$ is the logarithmic tangent sheaf associated with a complete intersection $Y\subset\PP^N$ in the sense of \cite{FJV}. Then we establish its general properties, including the following; see Proposition \ref{mainprop}.

\begin{itemize}
    \item Let $X$ and $Y$ be smooth hypersurfaces of $\PP^N$, and let $D=X\cap Y$ be a reduced and effective divisor of $X$. Then the net logarithmic tangent sheaf $\Tt_X (Y; \PP^N)$ fits into the exact sequence
\[
 0\to\Tt_X(Y;\PP^N)\to\Tt_{X}(-\log D)\to\mathcal{T}\!\mathrm{or}^1_{\PP^N}(\mathrm{coker}(\nabla),\Oo_X)\to 0.
\]
\end{itemize} 
Note that $\Tt_{X}(Y;\PP^N)$ also fits into the dual Poincar\'e residue sequence
\[
0\to\Tt_{X}(Y;\PP^N)\to\Tt_{X}\to\Nn_{X,Y}\to 0,
\]
where a sheaf $\Nn_{X,Y}$ is isomorphic to a twisted Jacobian ideal $\Jj_{D}(D)$ modulo torsion; see Remark \ref{rem1}. Hence, one can interpret $\Tt_X(Y;\PP^N)$ as a natural refinement of the usual logarithmic tangent sheaf.

Next, we investigate the $\mu$-stability of $\Tt_{X}(H;\PP^N)$ when $X$ is a hypersurface of degree $d\geq 2$ and $H$ is given by a hyperplane of $\PP^n$ with $X\cap H$ nonsingular; see Proposition \ref{stab}. 
\begin{itemize}
    \item For a smooth and very general hypersurface $X$ with $\deg(X)\geq 2$ and a hyperplane $H$ in $\PP^N$ such that $X\cap H$ is nonsingular, the bundle $\Tt_{X}(H;\PP^N)$ is $\mu$-stable with respect to an ample line bundle $\Hh=\Oo_{X}(1)$.   
\end{itemize}
To give a more explicit description on the closure of the set of such logarithmic vector bundles $\Tt_X(H;\PP^N)$ in the corresponding moduli space, we focus on the following two cases:
\begin{itemize}
    \item[(i)] $X=Q\subset\PP^3$ a smooth quadric surface;
    \item[(ii)] $X=S\subset\PP^3$ a smooth cubic surface.
\end{itemize}
In either case, we have a morphism from the dual projective space $(\PP^3)^{*}$ to the moduli space $\mathbf{M}_{X}^{\Hh}(c_1,c_2)$ defined by sending $[H]$ to $\Tt_{X}(H;\PP^3)$, by showing the following in Proposition \ref{stabbb} \& Proposition \ref{stabstab}.
\begin{itemize}
    \item $\Tt_{X}(H;\PP^3)$ is stable for \emph{every} hyperplane $H\in(\PP^3)^{*}$, including the instance where the intersection $X\cap H$ is singular subvariety of $X$.
\end{itemize}
Indeed, we show that the morphism is injective for each case, \ie the Torelli property holds for $\Tt_X(H;\PP^3)$; see Remark \ref {tor123} \& Corollary \ref{cinj2}. Furthermore, in case $\mathrm{(i)}$, the moduli space $\mathbf{M}_{Q}^{\Hh}((1,1),2)$ turns out to be isomorphic to $(\PP^3)^{*}$; see Corollary \ref{corr}. In case $\mathrm{(ii)}$, we obtain the following by Proposition \ref{char}.
\begin{itemize}
    \item The closure of the set of vector bundles $\Ff$ in $\mathbf{M}_{S}^{\Hh}(0,9)$ such that $\Ff(1)$ is globally generated with $\mathrm{h}^0(\Ff(1))=3$, is isomorphic to $(\PP^3)^{*}$.
\end{itemize}

\section{Preliminaries}
Let $Y=Y_1+ \dots+ Y_s$ be a reduced and effective divisor in $\PP^N$ with each $Y_j=\VV(G_j)$ irreducible for $G_j \in \CC [x_0, \dots, x_N]$ and $g_j=\deg (G_j)$. Assume further that $\cap_{j=1}^s Y_j$ is a complete intersection of codimension $s$ in $\PP^N$. The kernel of the gradient map 
\[
\xi_{Y}=\left(
\begin{array}{cc}
     \nabla{G_1}\\
     \vdots\\
     \nabla{G_s}
\end{array}
\right)
:\quad  
\Oo_{\PP^N}(1)^{\dirsum(N+1)} \longrightarrow \Dirsum_{j=1}^{s}\Oo_{\PP^N}(g_j)
\]
can be considered as a generalization of the logarithmic tangent sheaf, associated to the complete intersection $\cap_{j=1}^s Y_j$. We denoted it by $\Tt_{Y, \PP^N}:=\mathrm{ker}(\xi_{Y})$; refer \cite{FJV}. Note that the sheaf $\Tt_{Y, \PP^N}$ is reflexive with $\mathrm{rank}(\Tt_{Y, \PP^N}) =N-r+1$. 

Let $X=X_1 \cap \dots \cap X_r \subset \PP^N$ be a smooth complete intersection of dimension $\dim (X)=N-r$, defined by $r$ hypersurfaces $X_i=\VV (F_i)$ for homogeneous polynomials $F_1,\ldots,F_r\subset \CC[x_0,\ldots,x_N]$ with $f_i=\deg (F_i)$. We further assume that $D= X \cap Y$ be a complete intersection of codimension $r+s$ in $\PP^N$. Then one can consider the logarithmic tangent sheaf $\Tt_{X \cap Y, \PP^N}$, associated to the complete intersection $X \cap Y\subset \PP^N$, fitting into an exact sequence
\begin{equation}\label{mm1}
0\to \Tt_{X \cap Y, \PP^N} \to \Oo_{\PP^N}(1)^{\oplus (N+1)} \stackrel{\xi_{X\cap Y}}{\xrightarrow{\hspace*{1.5cm}}}\left(\Dirsum_{i=1}^{r}\Oo_{\PP^N}(f_i)\right)\dirsum\left(\Dirsum_{j=1}^{s}\Oo_{\PP^N}(g_j)\right), 
\end{equation}
where the map $\xi_{X \cap Y}$ is given by the Jacobian matrix
\[
\xi_{X\cap Y}:=
\left(
\begin{array}{cc}
     \nabla{F_1}\\
     \vdots\\
     \nabla{F_r}\\
     \nabla{G_1}\\
     \vdots\\
     \nabla{G_s}
\end{array}
\right)
=\left(
\begin{array}{ccc}
     \frac{\partial{F_1}}{\partial{x_0}}&\cdots&\frac{\partial{F_1}}{\partial{x_N}}\\
     \vdots&\ddots &\vdots\\
     \frac{\partial{F_r}}{\partial{x_0}}&\cdots&\frac{\partial{F_r}}{\partial{x_N}}\\
     \frac{\partial{G_1}}{\partial{x_0}}&\cdots&\frac{\partial{G_1}}{\partial{x_N}}\\
     \vdots&\ddots &\vdots\\
     \frac{\partial{G_s}}{\partial{x_0}}&\cdots&\frac{\partial{G_s}}{\partial{x_N}}
\end{array}
\right).
\]

\noindent Then $\Tt_{X \cap Y, \PP^N}$ is a reflexive sheaf on $\PP^N$ with $\mathrm{rank}(\Tt_{X \cap Y, \PP^N})=N-r-s+1$. Set
\[
\Aa_{X \cap Y}:=\mathrm{Im}(\xi_{X \cap Y}) , \quad \Bb_{X \cap Y}:=\mathrm{coker}(\xi_{X \cap Y})
\]
to get the following two exact sequences:

\begin{equation}\label{fe1}
    0\to \Tt_{X \cap Y, \PP^N} \to \Oo_{\PP^N}(1)^{\dirsum(N+1)} \to \Aa_{X\cap Y} \to 0
\end{equation}

\noindent and
\begin{equation}\label{fe2}
    0\to \Aa_{X \cap Y} \to \left(\Dirsum_{i=1}^{r}\Oo_{\PP^N}(f_i)\right)\dirsum\left(\Dirsum_{i=1}^{s}\Oo_{\PP^N}(g_i)\right) \to \Bb_{X \cap Y} \to 0.
\end{equation}
Tensoring \eqref{fe1} by $\Oo_X$, we get
\[
0 \to \mathcal{T}\!\mathrm{or}^1_{{\PP^N}}\left(\Aa_{X\cap Y}, \Oo_X\right) \stackrel{\alpha}{\longrightarrow} \Tt_{X \cap Y, \PP^N}\tensor \Oo_X
\to \Oo_{X}(1)^{\dirsum(N+1)} \to \Aa_{X \cap Y}\tensor{\Oo_{X}}\to 0.
\]

\begin{definition}
We set $\Tt_X(Y; \PP^N):=\mathrm{coker}(\alpha)$, i.e. 
\[
\Tt_X (Y; \PP^N):=\bigslant{\Tt_{X \cap Y, \PP^N}\tensor \Oo_X}{\mathcal{T}\!\mathrm{or}^1_{{\PP^N}}\left(\Aa_{X \cap Y}, \Oo_X\right)}, 
\]
and call it the {\it net logarithmic tangent sheaf} on $X$, associated to $Y$, or equivalently to $X \cap Y$. Then it is a torsion-free sheaf on $X$ with rank $\dim (X)-s+1$, not necessarily reflexive. Note that $\Tt_{\PP^N}(Y; \PP^N)\tensor\Oo_{\PP^N}(-1)$ with $r=0$ is the logarithmic sheaf on $\PP^N$ associated to a complete intersection $Y$ in the sense of \cite{FJV}. 
\end{definition}

\begin{remark}\label{depend}
Note that the sheaf $\Tt_{X \cap Y, \PP^N}$ depends only on the complete intersection ideal $(F_1, \dots, F_r, G_1, \dots, G_s)\subset \CC [x_0, \dots, x_N]$, not on the choice of its minimal generators. In particular, the sheaf $\Tt_X(Y; \PP^N)$ depends on $X$ and $X \cap Y$, i.e. for another complete intersection $Y'$ such that $X \cap Y=X \cap Y'$ we have $\Tt_X(Y; \PP^N) \cong \Tt_X(Y'; \PP^N)$.
\end{remark}

\begin{remark}\label{general}
Let $\PP^n\subset \PP^N$ be a linear subspace and assume that it is defined by $x_{n+1}=\dots = x_{N}=0$ by a change of variables. For a smooth hypersurface $Y=\VV (G)$ of degree $d$ in $\PP^N$, one can have the following exact sequence
\[
0\to \Tt_{\PP^n \cap Y} \to \Oo_{\PP^N}(1)^{\oplus (N+1)} \to \Oo_{\PP^N}(1)^{\oplus (N-n)}\oplus \Oo_{\PP^N}(d),
\]
which induces an exact sequence
\begin{equation}\label{eqa33}
0\to \Tt_{\PP^n \cap Y} \to \Oo_{\PP^N}(1)^{\oplus (n+1)}\stackrel{\mathrm{M}}{\longrightarrow} \Ii_{Z', \PP^N}(d) \to 0, 
\end{equation}
where $\mathrm{M}=\begin{pmatrix} \frac{\partial G}{\partial x_0} & \dots & \frac{\partial G}{\partial x_n}\end{pmatrix}$ and $Z'$ is the degeneracy locus of $\mathrm{M}$. By tensoring the sequence (\ref{eqa33}) by $\Oo_{\PP^n}$, one gets
\begin{equation}\label{eqa34}
0\to \Tt_{\PP^n}(Y; \PP^N) \to \Oo_{\PP^n}(1)^{\oplus (n+1)}\to \Ii_{Z', \PP^N}(d)\otimes \Oo_{\PP^n} \to 0.
\end{equation}
Note that the sheaf $\Ii_{Z', \PP^N}(d)\otimes \Oo_{\PP^n}$ is an extension of $\Ii_{Z' \cap \PP^n, \PP^n}(d)$ by $\Oo_{Z' \cap \PP^n}$ and so we get 
\[
\mathcal{H}om_{\PP^n}(\Ii_{Z', \PP^N}(d)\otimes \Oo_{\PP^n}, \Oo_{\PP^n}) \cong \Oo_{\PP^n}(-d).
\]
Applying the functor $\mathcal{H}om_{\PP^n}( -, \Oo_{\PP^n})$ to the sequence (\ref{eqa34}), we get 
\[
0\to \Oo_{\PP^n}(-d) \to \Oo_{\PP^n}(-1)^{\oplus (n+1)}\to \Tt_{\PP^n}(Y;\PP^N)^\vee \stackrel{\sigma}{\longrightarrow} \mathcal{E}xt_{\PP^n}^1(\Ii_{Z', \PP^n}(d)\otimes \Oo_{\PP^n}, \Oo_{\PP^n}) \to 0.
\]
Note that the net logarithmic sheaf $\mathring{\Omega}_{\PP^n}^1(\log (\PP^n \cap Y))$ is introduced in \cite{HMP}; indeed, the sheaf is identified with the kernel of the map $\sigma$. Hence we can say that the net logarithmic tangent sheaf $\Tt_X(Y; \PP^N)$ is the proper counterpart of the net logarithmic sheaf $\mathring{\Omega}^1_X(\log (X \cap Y))$ in the dual sense. Note that the net logarithmic sheaf $\mathring{\Omega}_X^1 (\log D)$ is defined over any smooth variety $X$ associated to any effective divisor $D$, and so it would be interesting to extend the notion $\Tt_X(Y; \PP^N)$ to non-complete intersection subvariety $X$. 
\end{remark}

\begin{remark}\label{linear}
As in Remark \ref{general}, consider a linear subspace $\PP^n \subset \PP^N$. Assume that the complete intersections $X$ and $Y$ both are contained in $\PP^n$, i.e. $\langle x_{n+1}, \dots, x_N\rangle \subset \mathrm{I}(X) \cap \mathrm{I}(Y)$. Then we have $\Tt_X(Y; \PP^n) \cong \Tt_X(Y; \PP^N)$ by the argument in Remark \ref{depend}. In particular, it is enough to consider the sheaf $\Tt_X(Y; \PP^N)$ when $X\cup Y$ is nondegenerate in $\PP^N$.  
\end{remark}


\begin{proposition}\label{mainprop}
Let $X, Y\subset\PP^N$ be two smooth hypersurfaces, not necessarily intersecting transversally, such that $D=X \cap Y$ is a reduced and effective divisor of $X$. Then we have an exact sequence
\begin{equation}\label{mm2}
 0\to\Tt_X(Y;\PP^N)\to\Tt_{X}(-\log D)\to\mathcal{T}\!\mathrm{or}^1_{\PP^N}(\Qq,\Oo_X)\to 0,
\end{equation}
where $\Tt_{X}(-\log D)$ is the dual of the sheaf of logarithmic differential forms $\Omega^1_X(\log D)$ and $\Qq:=\Bb_{X \cap Y}=\mathrm{coker}(\xi_{X \cap Y})$. In particular, we have
\[
\Tt_X(Y; \PP^N)^{\vee\vee}\cong \Tt_X(-\log D).
\]
\end{proposition}

\begin{proof}
Let $X=\VV(F)$ and $Y=\VV(G)$ for $F,G\in\CC[x_0,\cdots,x_N]$ and $f=\deg(F),\, g=\deg(G)$. We have an exact sequence for $\Tt_{X\cap Y,\PP^N}$ :
\[
0 \to \Tt_{X\cap Y,\PP^N} \to \Oo_{\PP^N}(1)^{\oplus{(N+1)}} \stackrel{\xi_{X\cap Y}}{\barrow} \Oo_{\PP^N}(f) \oplus \Oo_{\PP^N}(g)
\]
and then we get the following two commutative diagrams
\begin{equation}\label{Tfg1}
\begin{tikzcd}[
  row sep=normal, column sep=normal,
  ar symbol/.style = {draw=none,"\textstyle#1" description,sloped},
  isomorphic/.style = {ar symbol={\cong}},
  ]
  & & & 0\ar[d] &\\
  & 0\ar[d] &  & \Kk \ar[d] & \\
  0 \ar[r] & \Tt_{X\cap Y,\PP^N} \ar[d]\ar[r] & \left(\Oo_{\PP^N}(1)\right)^{\oplus(N+1)} \ar[d,isomorphic]\ar[r,"\xi_{X\cap Y}"] & \Aa_{X\cap Y} \ar[d] \ar[r] & 0  \\
  0 \ar[r] & \Tt_{Y,\PP^N} \ar[d]\ar[r] & \left(\Oo_{\PP^N}(1)\right)^{\oplus(N+1)} \ar[r,"\xi_Y"] & \Oo_{\PP^N}(g) \ar[d]\ar[r] & 0\\
  & \Kk \ar[d] &  & 0 & \\
  & 0 &  &  &
\end{tikzcd}
\end{equation}
and
\begin{equation}\label{Tfg2}
\begin{tikzcd}[
  row sep=normal, column sep=normal,
  ar symbol/.style = {draw=none,"\textstyle#1" description,sloped},
  isomorphic/.style = {ar symbol={\cong}},
  ]
  & 0\ar[d] & 0 \ar[d] &  \\
  0 \ar[r] & \Kk \ar[d]\ar[r] & \Oo_{\PP^N}(f) \ar[d]\ar[r] & \Bb_{X\cap Y} \ar[d,isomorphic] \ar[r] & 0  \\
  0 \ar[r] & \Aa_{X\cap Y} \ar[d]\ar[r] & \Oo_{\PP^N}(f)\oplus\Oo_{\PP^N}(g) \ar[d]\ar[r] & \Bb_{X\cap Y} \ar[r] & 0\\
 & \Oo_{\PP^N}(g)\ar[d]\ar[r,isomorphic] & \Oo_{\PP^N}(g)\ar[d] & \\
 & 0 & 0. & 
\end{tikzcd}
\end{equation}
from the associated short exact sequences; see \eqref{fe1} and \eqref{fe2}. Now, the first vertical sequence of \eqref{Tfg1} gives us that
\begin{equation}\label{Tfg3}
    \begin{tikzcd}[cramped,
  row sep=small, column sep=small,
  ar symbol/.style = {draw=none,"\textstyle#1" description,sloped},
  isomorphic/.style = {ar symbol={\cong}},
  ]
 0 \ar[r] & \mathcal{T}\!\mathrm{or}^1_{\PP^N}(\Kk,\Oo_X)\ar[d,isomorphic]\ar[r] &\Tt_{X\cap Y,\PP^N}\otimes\Oo_{X} \ar[dr,twoheadrightarrow]\ar[rr]&& \Tt_{Y,\PP^N}\otimes\Oo_{X} \ar[r] &\Kk\otimes\Oo_{X} \ar[r]& 0. \\
          &\mathcal{T}\!\mathrm{or}^1_{\PP^N}(\Aa_{X\cap Y},\Oo_{X})& & \Tt_X(Y;\PP^N)\ar[ur,hookrightarrow] & & &
\end{tikzcd}  
\end{equation}
\\
and from the first horizontal sequence of \eqref{Tfg2} we also have

\begin{equation}\label{Tfg4}
        \begin{tikzcd}[cramped,
  row sep=small, column sep=small,
  ar symbol/.style = {draw=none,"\textstyle#1" description,sloped},
  isomorphic/.style = {ar symbol={\cong}},
  ]
 0 \ar[r] & \mathcal{T}\!\mathrm{or}^1_{\PP^N}(\Bb_{X\cap Y},\Oo_X)\ar[r] &\Kk\otimes\Oo_{X} \ar[dr,twoheadrightarrow]\ar[rr]&& \Oo_{X}(f) \ar[r] &\Bb_{X\cap Y}\otimes\Oo_{X} \ar[r]& 0. \\
          & & & \Ss\ar[ur,hookrightarrow] & & &
\end{tikzcd}  
\end{equation}
\\
On the other hand, since  $\Tt_{Y,\PP^N}=\Tt_{\PP^N}(-\log Y)$ we have a short exact sequence
\[
0\to\Tt_{Y,\PP^N}\to\Tt_{\PP^N}\to\Oo_{Y}(Y)\to 0
\]
and by tensoring with $\Oo_{X}$, we get
\[
0\to\Tt_{\PP^N}(-\log Y)\tensor\Oo_X\to\left(\Tt_{\PP^N}\right)_{|X}\to\Oo_{Y}(Y)\tensor \Oo_X\cong\Oo_{D}(D)\to 0.
\]   
Note that $\mathcal{T}\!\mathrm{or}^1_{\PP^N}(\Oo_{Y}(Y),\Oo_X)=0$, since $X$ and $Y$ does not have a common irreducible component. Furthermore, we have an exact sequence
\[
0\to\Tt_{X}(-\log D)\to\Tt_{X}\to \Jj_{D}(D)\to 0;
\]
refer \cite[(2.1)]{dolgachev2007logarithmic}. Then, considering
\[
0\to \Jj_{D}(D)\to\Oo_{D}(D)\to\Oo_{D^s}\to 0
\]
where $D^s$ is the closed subscheme of $D$ defined by the sheaf of ideals $\Jj_{D}$ so that $\Oo_{D^s}=\Oo_{D} / \Jj_{D}\cong\Bb_{X\cap Y}\tensor\Oo_X$,
we can construct a commutative diagram
\begin{equation}\label{Tfg5}
\begin{tikzcd}[
  row sep=normal, column sep=normal,
  ar symbol/.style = {draw=none,"\textstyle#1" description,sloped},
  isomorphic/.style = {ar symbol={\cong}},
  ]
   &  0 \ar[d] & 0 \ar[d] & 0\ar[d] \\
    0\ar[r]& \Tt_{X}(-\log D) \ar[d]\ar[r] & \Tt_{X} \ar[d]\ar[r] & \Jj_{D}(D) \ar[d] \ar[r] & 0  \\
    0\ar[r] & \Tt_{Y,\PP^N}\otimes\Oo_{X} \ar[r]\ar[d] & \left(\Tt_{\PP^N}\right)_{|X} \ar[d]\ar[r] & \Oo_{Y}(Y)\otimes\Oo_{X}\cong\Oo_{D}(D) \ar[d]\ar[r] & 0\\
    0\ar[r] & \Ss \ar[r]\ar[d] & \Oo_{X}(f) \ar[r]\ar[d] &\Oo_{D^s}\cong\Bb_{X\cap Y}\otimes\Oo_{X}\ar[r]\ar[d] & 0 \\
   & 0 & 0 & 0. 
\end{tikzcd}
\end{equation}
Now $\eqref{Tfg3},\eqref{Tfg4}$ and $\eqref{Tfg5}$ allow us to construct a commutative diagram 

\begin{equation*}
\begin{tikzcd}[
  row sep=normal, column sep=normal,
  ar symbol/.style = {draw=none,"\textstyle#1" description,sloped},
  isomorphic/.style = {ar symbol={\cong}},
  ]
  & & 0 \ar[d] & 0 \ar[d] \\
  0 \ar[r]& \Tt_X(Y;\PP^N)\ar[d,isomorphic]\ar[r] & \Tt_{X}(-\log D) \ar[d]\ar[r] & \mathcal{T}\!\mathrm{or}^1_{\PP^N}(\Bb_{X\cap Y},\Oo_{X}) \ar[d] \ar[r] & 0  \\
 0\ar[r] & \Tt_X(Y;\PP^N) \ar[r] & \Tt_{Y,\PP^N}\otimes\Oo_{X} \ar[d]\ar[r] & \Kk\otimes\Oo_{X} \ar[d]\ar[r] & 0\\
 & & \Ss\ar[d]\ar[r,isomorphic] & \Ss\ar[d]\\
 & & 0 & 0,
\end{tikzcd}
\end{equation*}
and the first horizontal sequence yields the conclusion.
\end{proof}

\begin{remark}\label{rem1}
By Proposition \ref{mainprop}, one can consider an injective composition map 
\[
\iota~:~\Tt_{X}(Y;\PP^N)\rightarrow\Tt_{X}(-\log D) \rightarrow\Tt_{X}.
\]
Setting $\Nn_{X, Y}:=\mathrm{coker}(\iota)$, one can have the following commutative diagram
\begin{equation*}
\begin{tikzcd}[
  row sep=normal, column sep=normal,
  ar symbol/.style = {draw=none,"\textstyle#1" description,sloped},
  isomorphic/.style = {ar symbol={\cong}},
  ]
  & & & 0\ar[d] &\\
  & 0\ar[d] &  & \mathcal{T}\!\mathrm{or}^1_{\PP^N}(\Qq,\Oo_X) \ar[d] & \\
  0 \ar[r] & \Tt_{X}(Y;\PP^N) \ar[d]\ar[r] & \Tt_{X} \ar[d,isomorphic]\ar[r] & \Nn_{X, Y} \ar[d] \ar[r] & 0  \\
  0 \ar[r] & \Tt_{X}(-\log D) \ar[d]\ar[r] & \Tt_{X} \ar[r] & \Jj_{D}(D) \ar[d]\ar[r] & 0,\\
  & \mathcal{T}\!\mathrm{or}^1_{\PP^N}(\Qq,\Oo_X) \ar[d] &  & 0 & \\
  & 0, &  &  &
\end{tikzcd}
\end{equation*}
where the upper horizontal exact sequence can be considered as the Poincar\'e residue sequence for $\Tt_{X}(Y;\PP^N)$.
Note that if $D$ is nonsingular, then we have $\Nn_{X, Y}\cong \Oo_{D}(D)$ the normal bundle of $D$ in $X$.
\end{remark}

Let $X=\VV(F)\subset\PP^N$ be a smooth hypersurface of degree $d$ and $H=\VV(G)\subset\PP^N$ be a hyperplane. Then we have
\[
0\to\Tt_{X\cap H, \PP^N}\to \Oo_{\PP^N}(1)^{\dirsum(N+1)}\overset{\xi_{X\cap H}}{\barrow}\Oo_{\PP^N}(d)\dirsum\Oo_{\PP^N}(1),
\]
which induces an exact sequence
\begin{equation}\label{hypeq1}
    0\to\Tt_{X\cap H, \PP^N}\to \Oo_{\PP^N}(1)^{\dirsum N}\overset{\bar{\xi}_{X\cap H}}{\barrow}\Oo_{\PP^N}(d).
\end{equation}
Note that $\bar{\xi}_{X\cap H}$ is surjective if and only if $X\cap H$ is smooth. In general, we have $\mathrm{Im}(\bar{\xi}_{X\cap H})=\Ii_{Z,\PP^N}(d)$ for $Z=\textrm{Sing}(X\cap H)$. By tensoring \eqref{hypeq1} with $\Oo_X$, we get
\[
0\to\Tt_X(H;\PP^N)\to \Oo_X(1)^{\dirsum N}\to\Ii_{Z,\PP^N}(d)\otimes \Oo_X\to 0.
\]
From the natural map
\[
\Tt_X(-\log X\cap H)\to\left(\Tt_{\PP^N}(-\log H)\right)_{|X}\cong\Oo_{X }(1)^{\dirsum N}
\]
one can obtain a commutative diagram
\begin{equation*}\label{commm2}
\begin{tikzcd}[
  row sep=normal, column sep=normal,
  ar symbol/.style = {draw=none,"\textstyle#1" description,sloped},
  isomorphic/.style = {ar symbol={\cong}},
  ]
  & & & 0\ar[d] &\\
  & 0\ar[d] &  & \mathcal{T}\!\mathrm{or}^1_X(\Oo_Z,\Oo_X) \ar[d] & \\
  0 \ar[r] & \Tt_{X}(H;\PP^N) \ar[d]\ar[r] & \Oo_{X}(1)^{\oplus N} \ar[d,isomorphic]\ar[r] & \Ii_{Z,\PP^N}(d)\tensor\Oo_X \ar[d] \ar[r] & 0  \\
  0 \ar[r] & \Tt_{X}(-\log X\cap H) \ar[d]\ar[r] & \Oo_{X}(1)^{\oplus N} \ar[r] & \Ii_{Z,X}(d) \ar[d]\ar[r] & 0\\
  & \mathcal{T}\!\mathrm{or}^1_X(\Oo_Z,\Oo_X) \ar[d] &  & 0 & \\
  & 0. &  &  &
\end{tikzcd}
\end{equation*}
Recall that $\Oo_Z=\Bb_{X\cap H}=\Qq$ in Proposition \ref{mainprop}.

\begin{definition}\cite[Page 4]{MirRoig2007lec}
Let $X$ be a smooth irreducible projective variety of dimension $n$ and let $H$ be an ample divisor on $X$. For a torsion-free sheaf $\Ee$ on $X$ one sets
\[
\mu_{H}(\Ee):=\frac{c_1(\Ee)\cdot H^{n-1}}{\mathrm{rank}(\Ee)},\hspace{.5cm} \mathrm{P}_{\Ee}(t):=\frac{\chi(\Ee\tensor\Oo_{X}(tH))}{\mathrm{rank}(\Ee)}.
\]
The sheaf $\Ee$ is \textit{$\mu$-semistable} (resp. \textit{semistable}) with respect to the polarization $H$ if and only if
\[
\mu_{H}(\Ff)\leq\mu_{H}(\Ee)\hspace{.5cm}(\text{resp.}\;\mathrm{P}_{\Ff}(t)\leq \mathrm{P}_{\Ee}(t) \;\;\text{for}\;\;t\gg 0) 
\]
for all non-zero subsheaves $\Ff\subset\Ee$ with $\mathrm{rank}(\Ff)<\mathrm{rank}(\Ee)$; if strict inequality holds then $\Ee$ is \textit{$\mu$-stable} (resp. \textit{stable}) with respect to $H$. We often omit the subscript in $\mu_{H}$ if the polarization is clear in the context. 
\end{definition}

\begin{proposition}\label{stab}
Let $X$ be a smooth hypersurface in $\PP^N$ of degree $d\ge 2$; in case $N=3$ and $d\ge 4$ we further assume that $X$ is very general. For a hyperplane $H\subset \PP^N$ such that $X \cap H$ is smooth, the net logarithmic tangent sheaf $\Tt_X(H; \PP^N)$ is $\mu$-stable. 
\end{proposition}

\begin{proof}
Set $D=X \cap H$ and note that it is enough to show that $\Tt_X(-\log D)$ is $\mu$-stable by Proposition \ref{mainprop}. 

\noindent (a) \quad First assume that $\mathrm{Pic}(X)\cong \ZZ \langle \Oo_X(1) \rangle$, which is the case
\begin{itemize}
\item $N\ge 4$ by the Lefschetz theorem, or 
\item $N=3$, $d\ge 4$, and $X$ is very general; see \cite{Voisin}. 
\end{itemize}
By taking the dual of the second horizontal sequence in (\ref{commm2}), we get 
\begin{equation}\label{XHeq}
0\to\Oo_{X}(-d)\to\Oo_{X}(-1)^{\oplus N}\to\Tt_X(-\log D)^{\vee}\cong \Omega_X^1(\log D) \to 0.    
\end{equation}
Set $\Ee:=\Omega_X^1(\log D)\tensor\Oo_{X}(d)$. Applying \cite[Theorem 4.1.3]{hirzebruch1966topological} to (\ref{XHeq}), one gets
\begin{equation}\label{XHeq1}
0\to \left(\bigwedge^{p-1}\Ee\right)(m) \to \Oo_{X}(p{\cdot}(d-1)+m)^{\dirsum{\binom{N}{p}}} \to \left( \bigwedge^{p}\Ee\right) (m)\to 0
\end{equation}
for $1\leq p \leq N-1$ and $m\in \ZZ$. Now for the $\mu$-stability of $\Ee$, it suffices to show that 
\[
\mathrm{h}^0\left( \left(\bigwedge^{p}\Ee\right)(m)\right)=0
\]
for all $m\leq -p\cdot\mu(\Ee)$ and for all $1\leq p < N-1$ ; see \cite[Proposition 1.1]{bohnhorst} and \cite[Proposition 2.12]{MirRoig2007lec}.
\vspace{.1cm}




\noindent From \eqref{XHeq1} we have
\begin{equation}\label{eegg123}
\begin{array}{ll}
    \mathrm{h}^0\left(\left(\bigwedge^{p}\Ee\right)(m)\right)&=\mathrm{h}^0\left(\Oo_{X}(p{\cdot}(d-1)+m)^{\dirsum\binom{N}{p}}\right)-\mathrm{h}^0\left(\left(\bigwedge^{p-1}\Ee\right)(m)\right)+\mathrm{h}^1\left(\left(\bigwedge^{p-1}\Ee\right)(m)\right)\\[.3cm]
    &=\displaystyle\sum_{k=0}^{p}(-1)^{p-k}\binom{N}{k}\;\mathrm{h^0}\Big(\Oo_{X}(k\cdot(d-1+m))\Big) +\sum_{k=0}^{p-1}(-1)^{p-1-k}\;\mathrm{h}^1\left(\left(\bigwedge^{k}\Ee\right)(m)\right)\\[.6cm]
    &\displaystyle=\sum_{k=0}^{p}(-1)^{p-k}\binom{N}{k}\left[\binom{N+k\cdot(d-1)+m}{N}-\binom{N+k\cdot(d-1)-d+m}{N}\right]\\[.6cm]
    &\hspace{.3cm}\displaystyle+\sum_{k=0}^{p-1}(-1)^{p-1-k}\;\mathrm{h}^1\left(\left(\bigwedge^{k}\Ee\right)(m)\right)
\end{array}
\end{equation}
for $1\leq p<N-1$ and $m\in \ZZ$. On the other hand, again from $\eqref{XHeq1}$ we get inductively that
\begin{equation*}
\mathrm{h}^1\left(\left(\bigwedge^{q}\Ee\right)(m)\right)=\mathrm{h}^2\left(\left(\bigwedge^{q-1}\Ee\right)(m)\right)=\cdots=\mathrm{h}^q\left(\Ee(m)\right)=0
\end{equation*}
for $1\leq q < N-2$ and $m\in\ZZ$. Hence, from \eqref{eegg123}, we have $\mathrm{h}^0(\left(\bigwedge^{p}\Ee\right)(m))=0$ for all $m\leq p{\cdot}\mu(\Ee)=-p\cdot\frac{N(1-d)}{N-1}<p(1-d)$ and for all $1\leq p<N-1$. Therefore, $\Ee$ is $\mu$-stable.

\vspace{.3cm}

\noindent (b) \quad Now consider the cases $(N,d)\in \{  (3,2), (3,3)\}$. Since the $\mu$-stability of $\Tt_X(-\log D)$ with $d=2$ is already shown in \cite{BHM}, we assume that $d=3$, i.e. $X$ is a cubic surface, realized from a blow-up $\pi : X \rightarrow \PP^2$ at six points $Z=\{ p_1, \dots, p_6\}$ in general position.
Recall that the exceptional divisors $E_i:=\pi^{-1}(p_i)$ associated to the points $p_i$ together with $L$ a divisor class in $\pi^*\Oo_{\PP^2}(1)$, freely generate $\mathrm{Pic}(X)\cong \ZZ \langle L, E_1, \dots, E_6\rangle \cong \ZZ^{\oplus 7}$ with intersection numbers
\[
L^2=1,\: L.E_i=0,\: E_i.E_j=0 \:\mbox{ and }\: E_i^2=-1
\]
for each $i\ne j$. The anticanonical line bundle $\Oo_X(-K_X) \cong \Oo_X(3L-\sum_{i=1}^6 E_i)$ turns out to be very ample so that it induces the original embedding $\iota: X \hookrightarrow \PP^3$, and in particular, we have $\Oo_X(1)\cong \Oo_X(-K_X)$. The strategy of the proof follows the proof of \cite[Theorem 5.4]{HMPV}. Suppose that $\Tt_{X}(-\log D)$ is not $\mu$-stable and there is a destabilizing line bundle
\[
\Ll:=\Oo_X(aL+\sum_{i=1}^6 b_iE_i)\hookrightarrow \Tt_{X}(-\log D).   
\]
Then we have 
\begin{equation}\label{eq1}
3a+\sum_{i=1}^6 b_i=\mu\left(\Oo_X(aL+\sum_{i=1}^6 b_iE_i)\right)\geq\mu(\Tt_{X}(-\log D))=0.
\end{equation}
From \cite[Lemma 2.2]{HMPV}, for a generic $C\in|\Oo_X(2L-\sum_{i\in I}E_i)|$ where $I\subset\{1,\ldots,6\}$ with $|I|=4$, we have $|(D\cap C)_{\text{red}}| =2$ and then considering a commutative diagram
\begin{equation}\label{key}
\begin{tikzcd}[
  row sep=normal, column sep=normal,
  ar symbol/.style = {draw=none,"\textstyle#1" description,sloped},
  isomorphic/.style = {ar symbol={\cong}},
  ]
  & 0\ar[d] & 0 \ar[d] &  \\
  0 \ar[r] & \Oo_{C} \ar[d]\ar[r] & \left(\Tt_{X}(-\log D)\right)_{|C} \ar[d]\ar[r] & \Oo_{C} \ar[d,isomorphic] \ar[r] & 0  \\
  0 \ar[r] & \Tt_{C}\cong\Oo_{C}(2) \ar[d]\ar[r] & \left(\Tt_X\right)_{|C}\cong\Oo_{C}\oplus\Oo_{C}(2) \ar[d]\ar[r] & \Oo_{C} \ar[r] & 0\\
 & \Oo_{p}\dirsum\Oo_{q}\ar[d]\ar[r,isomorphic] & \left(\Oo_{D}(D)\right)_{|C}\ar[d] & \\
 & 0 & 0 & 
\end{tikzcd}
\end{equation}
we get that $\left(\Tt_X(-\log D)\right)_{|C}\cong\Oo_{C}^{\dirsum 2}$. This implies that $2a+\sum_{i\in I}b_i\leq 0$ and since the set $I$ is an arbitrary subset of $\{1,\ldots,6\}$, we have that
\begin{equation}\label{eq2}
 \left\{
\begin{array}{ll}
2a+b_1+b_2+b_3+b_4\leq 0;\\ 
2a+b_3+b_4+b_5+b_6\leq 0;\\ 
2a+b_1+b_2+b_5+b_6\leq 0.
\end{array}
\right.
\end{equation}
Hence, together with \eqref{eq1} we get that
\begin{equation}\label{eq3}
    3a+\sum_{i=1}^6 b_i=0. 
\end{equation}
Similarly, for a line $L_{ij}\in|\Oo_X(L-E_i-E_j)|$ with $1\leq i<j\leq 6$, by \cite[Lemma 2.2]{HMPV}  we have that $\left(\Tt_X(-\log D)\right)_{|L_{ij}}\cong\Oo_{L_{ij}}(-1)\dirsum\Oo_{L_{ij}}(1)$. Restricting the destabilizing line bundle to $L_{ij}$, we obtain that $a+b_i+b_j\leq 1$ for any $1\leq i < j \leq 6$. Moreover, since $3a+\sum_{i=1}^6 b_i=(a+b_1+b_2)+(2a+b_3+b_4+b_5+b_6)=0$ and since \eqref{eq2} we get that $a+b_1+b_2\geq 0$ as well. Indeed, in general 
\[
0 \leq a+b_i+b_j\leq 1\;\; \text{for any}\;\; 1\leq i<j\leq 6
\]
and thus $a+b_i+b_j=0$ for any $1\leq i<j\leq 6$ because of the equality $\eqref{eq3}$. Furthermore, for a curve $C\in|\Oo_X(L)|$ we have 
\[
\left(\Tt_X(-\log D)\right)_{|C}\cong\Oo_{C}(-1+e)\oplus\Oo_{C}(1-e)
\]
with $e\in\{0,1\}$ while $\left(\Oo_X(aL+\sum_{i=1}^6 b_iE_i)\right)_{|C}\cong\Oo_{C}(a)$. Hence, we have $a\leq 1$ and $\Ll \cong \Oo_X(-2kL+\sum_{i=1}^6 kE_i)$ for some nonnegative integer $k$. Restricting it to the line $\widehat{L}_i\in|\Oo_{X}(2L+E_i-\sum_{i=1}^6 E_i)|$, we get that 
\[
\Ll_{|\widehat{L}_i} \cong \left(\Oo_X(-2kL+\sum_{i=1}^6 kE_i)\right)_{|\widehat{L}_i}\cong\Oo_{\widehat{L}_i}(k)
\]
with $\left(\Tt_X(-\log D)\right)_{|\widehat{L}_i}\cong\Oo_{\widehat{L}_i}(-1)\dirsum\Oo_{\widehat{L}_i}(1)$ and this implies that $k\leq 1$, so thus $k\in\{0,1\}$. In the case of $k=0$, we would have $\mathrm{h}^0(\Tt_S(-\log D))>0$. Then this would imply that $\mathrm{h}^0(\Tt_S)>0$, which is not true; see \cite[Corollary 3.5]{huybrechts2023cubic}. Thus one can assume that $k=1$. Then one can consider $\Ll$ as a subsheaf of $\Tt_X$ and let $\Ll\otimes \Oo_X(D')$ be its saturation in $\Tt_X$. First assume that $D'$ is an effective divisor. In particular, its cokernel is torsion-free, say $\Ii_{Z'', X}\otimes \Gg$ for a $0$-dimensional subscheme $Z''$ and the line bundle $\Gg:=\Ll^{-1}\otimes \Oo_X(-K_X-D')$.  
Then we have a commutative diagram
\begin{equation}\label{stabdiag}
\begin{tikzcd}[
  row sep=normal, column sep=normal,
  ar symbol/.style = {draw=none,"\textstyle#1" description,sloped},
  isomorphic/.style = {ar symbol={\cong}},
  ]
  & 0\ar[d] & 0 \ar[d] & 0 \ar[d] \\
  0 \ar[r] & \Ll \ar[d]\ar[r] & \Tt_{X}(-\log D) \ar[d]\ar[r] & \Ii_{Z',X}\tensor\Ll^{-1} \ar[d] \ar[r] & 0  \\
  0 \ar[r] &  \Ll\tensor\Oo_{X}(D') \ar[d]\ar[r] & \Tt_{X} \ar[d]\ar[r] & \Ii_{Z'',X}\tensor\Gg \ar[d] \ar[r] & 0\\
  0 \ar[r] & \Ll\tensor\Oo_{D'}(D')\ar[d]\ar[r] & \Oo_{D}(D) \ar[d] \ar[r] & \Kk \ar[d]\ar[r] & 0\\
 & 0 & 0 & 0,
\end{tikzcd}
\end{equation}
where $Z'$ is a $0$-dimensional subschemes of $X$ and $\Kk\cong \Oo_W(D)$ for a subscheme $W\subset D$. Since $D$ is irreducible, we have $D'=D$ and $\Ll\tensor\Oo_{X}(D')\cong\Oo_{X}(L)$. From $\mu(\Tt_{X})=\frac{3}{2}$ while $\mu(\Oo_{X}(L))=3$, we get that the tangent bundle $\Tt_X$ is not $\mu$-stable, which is not true; see \cite{Fa}. Hence $D'$ is trivial and so $\Ll$ is a line subbundle of $\Tt_X$. Then we have an exact sequence
\begin{equation}\label{ext2}
0\to \Ll \to \Tt_X \to \Ii_{Z'', X}\otimes \Ll^{-1}\tensor\Oo_{X}(-K_X)\to 0,
\end{equation}
where $Z''$ is a 0-dimensional subscheme of $X$ with $|Z''|=7$. If we restrict the sequence $\eqref{ext2}$ to the line $L_{ij}$, we obtain that
\begin{equation*}
0\to\Oo_{L_{ij}}(e) \to \left(\Tt_X\right)_{|L_{ij}} \to \Oo_{L_{ij}}(1-e) \to 0    
\end{equation*}
with $e=|Z''\cap L_{ij}|$. Then the only possibility would be $e=2$, because $\left(\Tt_X\right)_{|L_{ij}}\cong\Oo_{L_{ij}}(-1)\dirsum\Oo_{L_{ij}}(2)$. Since there are exactly $15$ lines of type $L_{ij}$ with $1\le i<j\le 6$, we obtain $|Z''|\geq \frac{15}{3}\times 2=10$, which is impossible. Indeed, through each point in $Z''$ there pass at most three lines $L_{ij}$ and the point with three passing lines is called an \emph{Eckhart} point.
\end{proof}

\begin{remark}
It is still unknown whether the condition in Proposition \ref{stab} for $X$ to be `very general' can be discarded or not. The starting case to extend the result can be the Fermat surface $X_m=\VV (x_0^m+\dots+x_3^m)$ for $m\ge 4$, whose Picard number is described in \cite{Shioda}.
\end{remark}

\begin{example}\label{exQ}
Consider the net logarithmic tangent sheaf on $Q=\VV(x_0x_3-x_1x_2)$ associated to $H=\VV(x_3)$ in $\PP^3$, namely $\Tt_{Q}(H;\PP^3)$. we have the logarithmic tangent sheaf $\Tt_{Q\cap H,\PP^3}$ associated to a complete intersection $Q\cap H\subset\PP^3$ with an exact sequence
\[
0\to\Tt_{Q\cap H}\to \Oo_{\PP^3}(1)^{\dirsum{4}}\xrightarrow[\hspace*{1.5cm}]{\xi_{Q\cap H}} \Oo_{\PP^3}(2)\dirsum\Oo_{\PP^3}(1).
\]
Note that the map $\xi_{Q\cap H}$ can be represented by the matrix
\[
\begin{pmatrix}
x_3  &   -x_2 &   -x_1 &   x_0\\
0   &   0   &   0   &   1
\end{pmatrix}
\]
and then we have a commutative diagram
\begin{equation}\label{diagQH}
\begin{tikzcd}[
  row sep=normal, column sep=normal,
  ar symbol/.style = {draw=none,"\textstyle#1" description,sloped},
  isomorphic/.style = {ar symbol={\cong}},
  ]
         &                                     & 0 \ar[d] && 0 \ar[d] \\
 0\ar[r] & \Tt_{Q\cap H}\ar[d,isomorphic]\ar[r] & \Oo_{\PP^3}(1)^{\dirsum\;3} \ar[d]\ar[rr,"(x_3\;-x_2\;-x_1)"] && \Ii_{p,\PP^3}(2) \ar[d]\ar[r] &  0  \\
 0\ar[r] & \Tt_{Q\cap H}\ar[r]& \Oo_{\PP^3}(1)^{\dirsum\;4} \ar[d]\ar[rr,"\xi_{Q\cap H}"] && \Aa_{Q\cap H} \ar[d]\ar[r] & 0\\
         &                   & \Oo_{\PP^3}(1) \ar[d]\ar[rr,isomorphic] && \Oo_{\PP^3}(1)\ar[d]  & \\
         &                   &    0 && 0 &
\end{tikzcd}
\end{equation}
where $\Ii_{p,\PP^3}$ is the ideal sheaf of a point $p=[1:0:0:0]\in \PP^3$. Since $\Ext^1_{\PP^3}(\Oo_{\PP^3}(1),\Ii_{p,\PP^3}(2))=0$ we have $\Aa_{Q\cap H}\cong\Ii_{p,\PP^3}(2)\oplus\Oo_{\PP^3}(1)$, so $\mathcal{T}\!\mathrm{or}^1_{\PP^3}(\Aa_{Q\cap H},\Oo_{Q})\cong\mathcal{T}\!\mathrm{or}^1_{\PP^3}(\Ii_{p,\PP^3}(2),\Oo_{Q})$.
Therefore, we get a short exact sequence for $\Tt_{Q}(H;\PP^3)$ :

\begin{equation}
0\to\Tt_{Q}(H;\PP^3)\to\Oo_{Q}(1,1)^{\dirsum 3}\to\Ii_{p,\PP^3}(2)\tensor{\Oo_{Q}}\to 0.
\end{equation}
Now, we view the quadric surface $Q$ as a multiprojective space $\PP^1\times\PP^1$. Taking bihomogeneous coordinates $a_0,a_1,b_0,b_1$, the quadric surface $Q\subset\PP^3$ is the image of the Segre map 
\[
\PP^1\times\PP^1\ni [a_0:a_1,b_0:b_1] \longmapsto [a_0b_0:a_0b_1:a_1b_0:a_1b_1]\in \PP^3.
\]
We also write $\Oo_{Q}(a,b)$ for the line bundle of bidegree $(a,b)$ on $Q$. 
Considering a locally free resolution
\[
0 \to \Oo_{Q}(1,1) \to \Oo_{Q}(1,2)\dirsum\Oo_{Q}(2,1)\xlongrightarrow{(a_1\;\;b_1)}\Ii_{p,Q}(2,2)\to 0,
\]
we can construct a commutative diagram
\begin{equation*}\label{diag2}
\begin{tikzcd}[
  row sep=normal, column sep=normal,
  ar symbol/.style = {draw=none,"\textstyle#1" description,sloped},
  isomorphic/.style = {ar symbol={\cong}},
  ]
         & 0 \ar[d]         & 0\ar[d]                                           & 0\ar[d]\\
         &   \Kk \ar[d]     & \Oo_{Q}(1,0)\ar[d]   & \Oo_{p}\ar[d]\\
 0\ar[r] & \Tt_Q(H;\PP^3)\ar[d]\ar[r]       & \Oo_{Q}(1,1)^{\dirsum\;3}\ar[d,"\psi"]\ar[r] &\Ii_{p,\PP^3}(2)\tensor\Oo_{Q} \ar[d]\ar[r] &  0  \\
 0\ar[r] & \Oo_Q(1,1) \ar[d]\ar[r] & \Oo_{Q}(1,2)\dirsum\Oo_{Q}(2,1)  \ar[d]\ar[r, "(a_1\;\;b_1)"] & \Ii_{p,Q}(2,2) \ar[r]\ar[d] & 0\\
         & \Qq \ar[d]   &  \Oo_{A} \ar[d]& 0    & \\
         &  0                 &    0 &  &
\end{tikzcd}
\end{equation*}
with $
\psi =
\begin{pmatrix}
    0 & -b_0 & 0\\
    a_1 & 0 & -a_0
\end{pmatrix}
$ and $A\in|\Oo_{Q}(1,0)|$. Then the snake lemma gives an exact sequence
\[
0\to\Kk\to\Oo_{Q}(1,0)\to\Oo_{p}\xrightarrow{0}\Qq\to\Oo_{A}\to 0,
\]
so that $\Kk\cong\Ii_{p,Q}(1,0)$ and $\Qq\cong\Oo_{A}$.
Thus we have
\begin{equation}\label{eqTQH1}
0 \to \Ii_{p,Q}(1,0)\to \Tt_Q(H;\PP^3) \to \Oo_{Q}(0,1)\to 0.   
\end{equation}
Note that this is not a split exact sequence. Furthermore, if we replace the map $\psi$ by 
$\psi'=\begin{pmatrix}
    b_1 & -b_0 & 0\\
    0 & 0 & -a_0
\end{pmatrix},$ it yields another exact sequence for $\Tt_{Q}(H;\PP^3)$ :
\begin{equation}\label{eqTQH2}
    0 \to \Ii_{p,Q}(0,1)\to \Tt_Q(H;\PP^3) \to \Oo_{Q}(1,0)\to 0,    
\end{equation}
which does not split too. On the other hand, since $Q\cap H$ is a pair of lines $A+B\subset Q$ with $A\in|\Oo_{Q}(1,0)|$ and $B\in|\Oo_Q(0,1)|$, we have $\Tt_{Q}(-\log Q\cap H)\cong \Oo_{Q}(1,0)\oplus\Oo_{Q}(0,1)$; refer to \cite[Proposition 6.2]{BHM}. Now considering the following commutative diagram
\begin{equation*}\label{diag3}
\begin{tikzcd}[
  row sep=normal, column sep=normal,
  ar symbol/.style = {draw=none,"\textstyle#1" description,sloped},
  isomorphic/.style = {ar symbol={\cong}},
  ]
         & 0 \ar[d]         & 0\ar[d]                                           & \\
 0\ar[r] & \Ii_{p,Q}(1,0) \ar[d]\ar[r]     & \Tt_{Q}(H;\PP^3)\ar[d]\ar[r]   & \Oo_{Q}(0,1) \ar[d,isomorphic]\ar[r] & 0\\
 0\ar[r] & \Oo_{Q}(1,0) \ar[d]\ar[r]   & \Tt_{Q}(-\log Q\cap H) \ar[d]\ar[r] &\Oo_{Q}(0,1) \ar[r] &  0  \\    
         & \Oo_{p} \ar[d]\ar[r,isomorphic]   &  \mathcal{T}\!\mathrm{or}^1_{\PP^3}(\Bb_{Q\cap H},\Oo_{Q}) \ar[d]&     & \\
         &  0                 &  0   &  &
\end{tikzcd}
\end{equation*}
we get the following short exact sequence
\begin{equation}\label{eqTQH3}
0\to \Tt_{Q}(H;\PP^3)\to\Oo_{Q}(1,0)\dirsum\Oo_{Q}(0,1)\to \Oo_p \to 0.    
\end{equation}
Note that $c_1(\Tt_{Q}(H;\PP^3))=(1,1)$ and $c_2(\Tt_{Q}(H;\PP^3))=2$.
\end{example}

\begin{proposition}\label{stabbb}
    For any hyperplane $H\subset \PP^3$, the net logarithmic tangent sheaf $\Tt_Q(H;\PP^3)$ is stable with respect to an ample line bundle $\Hh=\Oo_{Q}(1,1)$.
\end{proposition}
\begin{proof}
If $H\cap Q$ is a smooth conic, then we have $\Tt_Q(H; \PP^3) \cong \Tt_Q(-\log H \cap Q)$ and it is stable; see \cite[Proposition 2.10]{huh2011} and \cite[Remark 4.7]{BHM}. Now assume that $H\cap Q$ is a singular conic with the vertex $p\in Q$. Then the sheaf $\Tt_{Q}(H;\PP^3)$ is torsion-free, but not locally-free with the singularity at $p$. Pick a rank one subsheaf of $\Tt_{Q}(H;\PP^3)$, say $\Ii_{Z,Q}(a,b)$ for a $0$-dimensional subscheme $Z\subset Q$ with $(a,b)\in\ZZ^{\dirsum 2}$. Then we get an exact sequence
\begin{equation}\label{eqQH1}
    0\to \Ii_{Z,Q}(a,b) \to \Tt_{Q}(H;\PP^3)\to \Ii_{Z',Q}(1-a,1-b)\to 0,
\end{equation}
where $Z'$ is a $0$-dimensional subschemes of $Q$ with $a+b-2ab+|Z|+|Z'|=2$. Setting $\Ee:=\Tt_{Q}(H;\PP^3)$ and $\Ff:=\Ii_{Z,Q}(a,b)$, we need to show that 
\begin{equation} \label{eqQH2}
\mathrm{P}_{\Ff}(t)=t^2+(a+b+2)t+(a+1)(b+1)-|Z| < t^2+3t+\frac{3}{2}=\mathrm{P}_{\Ee}(t).
\end{equation}
for $t\gg 0$. 
We begin with the composition map
\[
g:\Ii_{Z,Q}(a,b)\hookrightarrow \Tt_{Q}(H;\PP^3)\rightarrow\Oo_{Q}(0,1)
\]
with the surjection in \eqref{eqTQH1}. Then there are two possibilities:

\vspace{.1cm}

\noindent (a) \quad 
If $g$ is trivial, then we have an injection $\Ii_{Z,Q}(a,b)\hookrightarrow\Ii_{p,Q}(1,0)$. This implies that $a\leq 1$, $b\leq 0$, and $|Z|\geq 1$. When $(a,b)\neq (1,0)$, \eqref{eqQH2} holds. Suppose $(a,b)=(1,0)$. Then from \eqref{eqQH1}, we have $|Z|=1$. In other words, $\Ii_{Z,Q}(a,b)\cong\Ii_{p,Q}(1,0)$. In this case, we have
\[
\mathrm{P}_{\Ff}(t) = t^2+3t+1
\]
and thus \eqref{eqQH2} holds.

\vspace{.1cm}

\noindent (b) \quad
If $g$ is not trivial, then we have $\Ii_{Z,Q}(a,b)\hookrightarrow\Oo_{Q}(0,1)$ and thus $a\leq 0$, $b\leq 1$. When $(a,b)\neq(0,1)$, \eqref{eqQH2} holds similarly as in (a). If $(a,b)=(0,1)$, then from \eqref{eqQH1} we get $|Z|\in \{0,1\}$, i.e. $\Ii_{Z,Q}(a,b)\cong\Oo_{Q}(0,1)$ or $\Ii_{p,Q}(0,1)$. In the former case \eqref{eqTQH1} would split, impossible.
If $\Ii_{Z,Q}(0,1)\cong\Ii_{p,Q}(0,1)$, then we get $\mathrm{P}_{\Ff}(t) = t^2+3t+1$ as before and this satisfies \eqref{eqQH2}. Therefore, $\Tt_{Q}(H;\PP^3)$ is stable.
\end{proof}

\begin{remark}\label{tor123}
Note that, for two distinct hyperplanes $H$ and $H'$, we have $\Tt_Q(H; \PP^3) \not\cong \Tt_Q(H';\PP^3)$. Indeed, for $H$ such that $H\cap Q$ is singular with the vertex point $p$, we have $H=T_pQ$ the tangent space of $Q$ at $p$. The case when $H \cap Q$ is smooth, is dealt in \cite[Corollary 4.6]{BHM}.    
\end{remark}

\begin{corollary}\label{corr}\cite[Proposition 2.10]{huh2011}
The moduli space $\mathbf{M}_Q^{\Hh}((1,1),2)$ of stable sheaves of rank two with the Chern classes $c_1=(1,1)$ and $c_2=2$ on $Q$ with respect to $\Hh=\Oo_Q(1,1)$, is isomorphic to $(\PP^3)^*$.
\end{corollary}

\begin{proof}
It suffices to show the existence of the sheaf $\Pp$ on $Q\times (\PP^3)^*$ such that 
\begin{equation}\label{fibre}
\Pp_{~|Q\times \{H\}} \cong \Tt_Q(H;\PP^3)
\end{equation}
for each $H\in (\PP^3)^*$, which would induce a morphism $\Psi : (\PP^3)^* \rightarrow \mathbf{M}_Q^{\Hh}((1,1),2)$ by Proposition \ref{stabbb}. Then it is automatically an isomorphism due to Remark \ref{tor123}. 

Set $\PP:=\PP^3\times (\PP^3)^*$ and $\widetilde{H}\subset \PP$ the universal divisor for the hyperplanes in $\PP^3$. Then we have $\widetilde{Q}:=Q\times (\PP^3)^* \in |\Oo_{\PP}(2,0)|$ and $\widetilde{H}\in |\Oo_{\PP}(1,1)|$. Let $\widetilde{F}$ and $\widetilde{G}$ be the defining equations for $\widetilde{Q}$ and $\widetilde{H}$, respectively. Then we have an exact sequence
\[
0\to \widetilde{\Tt} \to \Oo_{\PP}(1,0)^{\oplus 4} \stackrel{\nabla_1}{\longrightarrow} \Oo_{\PP}(2,0)\oplus \Oo_{\PP}(1,1)
\]
for some torsion-free coherent sheaf $\widetilde{\Tt}:=\mathrm{ker}(\nabla_1)$ of rank two on $\PP$, where the map $\nabla_1=\left(\nabla_1(\widetilde{F}), \nabla_1(\widetilde{G})\right)$ is defined to be the gradients of $\widetilde{F}$ and $\widetilde{G}$ with respect to the coordinates of $\PP^3$. Set $\Aa:=\mathrm{Im}(\nabla_1)$ and define 
\[
\Pp:=\bigslant{\widetilde{\Tt}\otimes \Oo_{\widetilde{Q}}}{\mathcal{T}\!\mathrm{or}^1_{\PP}\left(\Aa, \Oo_{\widetilde{Q}}\right)}.
\]
Then the sheaf $\Pp$ is a torsion-free, coherent sheaf of rank two on $\widetilde{Q}$, fitting into an exact sequence
\begin{equation}\label{eqaaa24}
0\to \Pp \to \Oo_{\widetilde{Q}}((1,1),0)^{\oplus 4} \to \Aa \otimes \Oo_{\widetilde{Q}} \to 0.
\end{equation}
For a fixed hyperplane $H\subset \PP^3$, tensor (\ref{eqaaa24}) by $\Oo_{Q\times \{H\}}$ with the identification $Q\times \{H\} \cong Q$ to have
\begin{equation}\label{univeq}
0\to \mathcal{T}\!\mathrm{or}^1_{\PP}\left( \Aa\otimes \Oo_{\widetilde{Q}}, \Oo_{Q\times \{H\}} \right) \to  \Pp_{~|Q\times \{H\}}\to \Oo_Q(1,1)^{\oplus 4} \to \Aa\otimes \Oo_{Q\times \{H\}} \to 0.
\end{equation}
For the sheaf $\Aa_{Q \cap H}$ in (\ref{fe1}), we have $\Aa\otimes \Oo_{Q \times \{H\}}\cong \Aa_{Q \cap H} \otimes \Oo_Q$. Thus it remains to show the vanishing of $\mathcal{T}\!\mathrm{or}^1_{\PP}\left( \Aa\otimes \Oo_{\widetilde{Q}}, \Oo_{Q\times \{H\}} \right)$, because the sequence (\ref{univeq}) would then be identical to (\ref{fe1}) and we can obtain the isomorphism (\ref{fibre}). From the exact sequence
\[
0\to \Aa \stackrel{\iota}{\longrightarrow} \Oo_{\PP}(2,0)\oplus \Oo_{\PP}(1,1) \to \Bb \to 0
\]
with $\Bb:=\mathrm{coker}(\iota)$, we get an exact sequence 

\begin{equation}\label{eeeqq23}
        \begin{tikzcd}[cramped,
  row sep=small, column sep=small,
  ar symbol/.style = {draw=none,"\textstyle#1" description,sloped},
  isomorphic/.style = {ar symbol={\cong}},
  ]
 0 \ar[r] & \mathcal{T}\!\mathrm{or}^1_{\PP}\left( \Bb, \Oo_{\widetilde{Q}}\right)\ar[r] &\Aa \otimes \Oo_{\widetilde{Q}} \ar[dr,twoheadrightarrow]\ar[rr, "f"]&& \Oo_{\widetilde{Q}}((2,2),0)\oplus \Oo_{\widetilde{Q}}((1,1), 1) \ar[r] &\Bb \otimes \Oo_{\widetilde{Q}} \ar[r]& 0 \\
          & & & \Mm\ar[ur,hookrightarrow] & & &
\end{tikzcd}      
\end{equation}
\noindent by tensoring with $\Oo_{\widetilde{Q}}$. We set $\Mm:=\mathrm{Im}(f)$. Indeed, the composition $\Aa \rightarrow \Oo_{\PP}(2,0)\oplus \Oo_{\PP}(1,1) \rightarrow \Oo_{\PP}(1,1)$ is surjecitve and so one can see that $\Aa$ fits into an exact sequence
\[
0\to \Ii_{\Delta, \PP}(2,0) \to \Aa \to \Oo_{\PP}(1,1) \to 0
\]
for a subscheme $\Delta\subset \PP$ given as the degeneracy locus of $\nabla_1$, and we get $\Bb \cong \Oo_\Delta(2,0)$. For example, if $\widetilde{F}=\sum_{i=0}^3 x_i^2$ and $\widetilde{G}=\sum_{i=0}^3 x_iy_i$, then $\Delta$ is the diagonal in $\PP$. Now from a commutative diagram
\begin{equation}
\begin{tikzcd}[
  row sep=normal, column sep=normal,
  ar symbol/.style = {draw=none,"\textstyle#1" description,sloped},
  isomorphic/.style = {ar symbol={\cong}},
  ]
  & 0\ar[d] & 0 \ar[d] &  \\
  &\mathcal{T}\!\mathrm{or}^1_{\PP}\left(\Oo_{\Delta}(2,0),\Oo_{\widetilde{Q}}\right) \ar[r,isomorphic]\ar[d] &\mathcal{T}\!\mathrm{or}^1_{\PP}\left(\Oo_{\Delta}(2,0),\Oo_{\widetilde{Q}}\right) \ar[d] \\
  0 \ar[r] & \Ii_{\Delta,\PP}(2,0)\tensor\Oo_{\widetilde{Q}} \ar[d]\ar[r] & \Aa\tensor\Oo_{\widetilde{Q}} \ar[d]\ar[r] & \Oo_{\widetilde{Q}}(1,1) \ar[d,isomorphic] \ar[r] & 0  \\
  0 \ar[r] & \Ii_{\Delta\cap\widetilde{Q},\widetilde{Q}}(2,0)\ar[d]\ar[r] & \Mm \ar[d] \ar[r] & \Oo_{\widetilde{Q}}(1,1) \ar[r] & 0 \\
 & 0 & 0 & 
\end{tikzcd}
\end{equation}
we have
\begin{equation*}
\mathcal{T}\!\mathrm{or}^i_{\widetilde{Q}}\left(\Mm,\Oo_{Q\times\{H\}}\right)\cong\mathcal{T}\!\mathrm{or}^i_{\widetilde{Q}}\left(\Ii_{\Delta\cap\widetilde{Q},\widetilde{Q}}(2,0),\Oo_{Q\times\{H\}}\right)\cong\mathcal{T}\!\mathrm{or}^{i+1}_{\widetilde{Q}}\left(\Oo_{\Delta\cap\widetilde{Q},\widetilde{Q}}(2,0),\Oo_{Q\times\{H\}}\right),
\end{equation*}
which vanishes for all $i\geq 1$ by \cite[Theorem 4 at page 110]{Serre}, because we have
\begin{equation*}
\left(\Delta\cap\widetilde{Q}\right)\cap \left(Q\times\{H\}\right)=\left\{
    \begin{array}{cc}
        \{(p,H)\} & \text{if $H=T_{p}H$}\\
        \emptyset & \text{otherwise.}
    \end{array}
    \right.
\end{equation*}
This implies that
\[
\mathcal{T}\!\mathrm{or}^1_{\widetilde{Q}}\left(\Aa\otimes\Oo_{\widetilde{Q}},\Oo_{Q\times\{H\}}\right)\cong\mathcal{T}\!\mathrm{or}^1_{\widetilde{Q}}\left(\mathcal{T}\!\mathrm{or}^1_{\PP}\left(\Oo_{\Delta}(2,0),\Oo_{\widetilde{Q}}\right),\Oo_{Q\times\{H\}}\right)=0,
\]
because $\mathcal{T}\!\mathrm{or}^1_{\PP}\left(\Oo_{\Delta}(2,0),\Oo_{\widetilde{Q}}\right)$ supports $\Delta\cap\widetilde{Q}$, which intersects $Q\times\{H\}$ transversely.
\end{proof}


\begin{remark}\label{example1}
On the contrary to example \ref{exQ}, one can restrict the sheaf $\Tt_{Q \cap H, \PP^3}$ to $H$, instead of $Q$, to form $\Tt_H(Q; \PP^3)$. In case when $Q \cap H$ is a smooth conic, then we have
$\Tt_H(Q; \PP^3)\cong \Tt_H(-\log Q \cap H)\cong \Tt_H(-1)$; see \cite{An2014log}. Assume that $Q \cap H$ is a singular conic with the vertex $p$. Then from $\eqref{diagQH}$ we have an exact sequence
\[
0\to\Tt_{H}(Q;\PP^3)\to\Oo_{H}(1)^{\dirsum{3}}\to\Ii_{p,\PP^3}(2)\tensor\Oo_{H}\to 0
\]
and then one can construct a commutative diagram
\begin{equation*}
\begin{tikzcd}[
  row sep=normal, column sep=normal,
  ar symbol/.style = {draw=none,"\textstyle#1" description,sloped},
  isomorphic/.style = {ar symbol={\cong}},
  ]
         &    0\ar[d]                    & 0 \ar[d] &&& 0\ar[d]\\
 0\ar[r] & \Ii_{p,H}(1)\ar[d]\ar[r]      & \Oo_{H}(1)\ar[d]\ar[rrr] &&& \Oo_{p} \ar[d]\ar[r] &  0  \\
 0\ar[r] & \Tt_{H}(Q;\PP^3) \ar[d]\ar[r] & \Oo_{H}(1)^{\dirsum\;3}\ar[d]\ar[rrr, "(x_3\;\;-x_2\;\;-x_1)\tensor \textbf{1}_{H}"] &&& \Ii_{p,\PP^3}(2)\tensor\Oo_{H} \ar[r]\ar[d] & 0\\
    0\ar[r]     & \Oo_{H} \ar[d]\ar[r]   &  \Oo_{H}(1)^{\dirsum\;2} \ar[rrr,"(-x_2\;\; -x_1)"]\ar[d] &&& \Ii_{p,H}(2)\ar[r]\ar[d]  & 0 \\
         &  0                 &    0 &&&  0.&
\end{tikzcd}
\end{equation*}
The first vertical sequence of the diagram with the fact that $\Ext^1_{\PP^2}(\Ii_{p,\PP^2}(1),\Oo_{\PP^2})=0$ implies that
\[
\Tt_{H}(Q;\PP^3)\cong \Ii_{p,H}(1)\dirsum\Oo_H\cong\Ii_{p,\PP^2}(1)\dirsum\Oo_{\PP^2}.
\]
Indeed, we know from \cite{dolgachev1993arrangements} that $\Tt_{H}(-\log Q\cap H)\cong \Oo_{H}(1)\oplus\Oo_H$ and so we can obtain an exact sequence
\[
 0\to\Tt_{H}(Q;\PP^3)\to\Tt_{H}(-\log Q\cap H)\to\Oo_{p}\cong\mathcal{T}\!\mathrm{or}^1_{\PP^3}(\Oo_{p},\Oo_X)\to 0,
\]
which is the sequence (\ref{mm2}). 
\end{remark}



\section{Cubic surface}
In this section, we assume that $X=S=\VV(F)$ is a smooth cubic surface with the blow-down map $\pi : S \rightarrow \PP^2$. Recall that $Z=\{p_1, \dots, p_6\}\subset \PP^2$ is the set of blown-up points in general position with the exceptional divisors $E_i:=\pi^{-1}(p_i)$. We set $L$ a divisor class in $\pi^*\Oo_{\PP^2}(1)$ and $\Hh:=\Oo_S(1)\cong \Oo_S(-K_S)$. For a coherent sheaf $\Ee$ on $S$ of rank $r$ with Chern classes $c(\Ee)=(c_1(\Ee), c_2(\Ee))\in \ZZ^{\oplus 7}\oplus \ZZ$, we have
\[
\chi(\Ee(t))=\left(\frac{3r}{2}\right) t^2+\left( \frac{3r+2c_1(\Ee)\cdot\Hh}{2} \right) t + \left (r+\frac{c_1(\Ee){\cdot} \Hh}{2}+\frac{c_1(\Ee)^2-2c_2(\Ee)}{2} \right)
\]
with respect to the ample line bundle $\Hh$. In particular, for $r=2$ and $c_1(\Ee)=0$, we get $\chi(\Ee(t))=3t^2+3t+2-c_2(\Ee)$. 

\begin{proposition}\label{tor}
    For any two distinct smooth hyperplane sections $D_1$ and $D_2$ in $|\Oo_S(1)|$, we have
    \[
    \Tt_S(-\log D_1) \ncong \Tt_S(-\log D_2).
    \]
\end{proposition}
\begin{proof}
Let $D$ be a fixed curve in the family of smooth hyperplane sections and pick a curve $C\in|\Oo_{S}(L)|$. Since $D$ is a curve from $|\Oo_{S}(3L-\sum_{i=1}^{6}E_i)|$, we get $D.C=3$ with the following three cases:
\[
\text{(a)}~ |C\cap D|_{\text{red}}=3,\quad \text{(b)}~ |C\cap D|_{\text{red}}=2,\quad \text{(c)}~ |C\cap D|_{\text{red}}=1.
\]
Again by \cite[Lemma 2.2]{HMPV}, we have
\[
\left(\Tt_S(-\log D)\right)_{|C}\cong
\begin{cases}
    \Oo_{C}\oplus\Oo_{C} & \text{for (a) or (b)}\\
    \Oo_{C}(-1)\oplus\Oo_{C}(1) & \text{for (c)},
\end{cases}
\]
from which we can find the nine inflection points of the cubic curve $\pi(D)\subset \PP^2$. It is well-known classically that a smooth cubic curve can be uniquely determined by the given nine inflection points. Therefore, the assertion follows immediately.
\end{proof}

\begin{corollary}\label{cinj1}
Let $\mathbf{M}_S^{\Hh}(c_1,c_2)$ be the moduli space of semistable sheaves of rank two on $S$ with Chern classes $(c_1, c_2)$ with respect to the polarization $\Hh$. Then the rational map 
\[
\Phi : \PP \mathrm{H}^0(\Oo_{\PP^3}(1)) \cong \PP\mathrm{H}^0(\Oo_S(-K_S))\cong\PP^3\dashrightarrow\mathbf{M}_S^{\Hh}(0,9),
\]
defined by sending $H\in |\Oo_{\PP^3}(1)|$ to the corresponding net logarithmic tangent sheaf $\Tt_S(H; \PP^3)$, is generically injective.
\end{corollary}

\begin{proof}
The assertion follows automatically from Propositions \ref{stab} and \ref{tor}. 
\end{proof}

\noindent Let $H=\VV(G)\subset \PP^3$ be a hyperplane, and then we get an exact sequence
\[
0 \to \Tt_{S\cap H, \PP^3} \to \Oo_{\PP^3}(1)^{\oplus{4}} \stackrel{\xi_{S\cap H}}{\barrow} \Oo_{\PP^3}(3) \oplus \Oo_{\PP^3}(1) \to \Aa_{S \cap H} \to 0.
\]
As in Example \ref{example1}, we have $\Aa \cong \Kk\oplus \Oo_{\PP^3}(1)$ and so we get an exact sequence
\begin{equation}\label{eeeq1}
0 \to \Tt_{S\cap H, \PP^3} \to \Oo_{\PP^3}(1)^{\oplus{3}} \stackrel{\overline{\xi}_{S\cap H}}{\barrow} \Kk \to 0.
\end{equation}
Let $\CC[x_0, x_1, x_2, x_3]$ be the coordinate ring of $\PP^3$. By changing of variable, we may assume that $H=V(x_3)$ and 
\[
\overline{\xi}:=\overline{\xi}_{S\cap H}=\begin{pmatrix} \frac{\partial F}{\partial x_0} & \frac{\partial F}{\partial x_1} & \frac{\partial F}{\partial x_2} \end{pmatrix}.
\]
Then we have $\Kk \cong \Ii_{R_0, \PP^3}(3)$ with $R_0$ the degeneracy locus of $\overline{\xi}$, i.e. $R_0=\VV \left(\frac{\partial F}{\partial x_0}, \frac{\partial F}{\partial x_1}, \frac{\partial F}{\partial x_2} \right)$. Since $S$ is smooth, the subscheme $R_0$ is $0$-dimensional with length $l (R_0)=8$. Now, tensoring (\ref{eeeq1}) with $\Oo_S$, one gets 
\begin{equation}\label{eeqq09}
0\to \Tt_S(H; \PP^3) \to \Oo_S(1)^{\oplus 3} \to \Ii_{R_0, \PP^3}(3)\otimes \Oo_S \to 0.
\end{equation}
For the scheme-theoretic intersection $R:=R_0 \cap S$, we get an extension 
\[
0\to \Oo_R \to \Ii_{R_0, \PP^3}(3)\otimes \Oo_S \to \Ii_{R, S}(3) \to 0, 
\]
from which we obtain (\ref{mm2}):
\[
0\to \Tt_S(H; \PP^3) \to \Tt_S(-\log D) \to \Oo_{R} \to 0,
\]
with $D=S \cap H$. 

\begin{remark}
For $\Ee:=\Tt_S(H; \PP^3)$ we get $(c_1, c_2)=(0,9)$ from (\ref{eeqq09}) for any $H\in (\PP^3)^*$ and so $\mathrm{P}_{\Ee}(t)=\frac{3}{2}t(t+1)-\frac{7}{2}$. In particular, the stability and semistability of $\Ee$ coincide.    
\end{remark}

\noindent For a hyperplane $H\subset \PP^3$ with $D=S \cap H$ singular, one can notice that $D$ is reduced. Thus we have the following $6$ types for $D\subset H$:
\begin{itemize}
\item[(a)] a nodal cubic curve with the nodal point $q$; 
\item [(b)] a cuspidal cubic curve with the cusp $q$;
\item [(c1)] $C'+L'$ for a smooth conic $C'$ and a line $L'$ with the two intersection points $q_1\ne q_2$;
\item [(c2)] $C'+L'$ for a smooth conic $C'$ and a tangent line $L'$ with a tangent point $q$. 
\item [(d1)] three distinct lines $L_1+L_2+L_3$ with three vertices $\{p_{12}, p_{13}, p_{23}\}$, $p_{ij}=L_i \cap L_j$;
\item [(d2)] three distinct lines $L_1+L_2+L_3$ with the triple point $q$. 
\end{itemize}
Setting $\{q_1, \dots, q_t\}\subset D$ the support of $R$ with multiplicity $\mu_j$ such that $\mu_1 \geq \dots \geq \mu_t$, one can consider $\mu(R):=(\mu_1, \dots, \mu_t) \in \NN^{\oplus t}$ to encode the information about the singularity of $D$. Then we have the following table. 
\vspace{.2cm}
\begin{center}
\begin{tabular}{|c||c|c|c|c|c|c|}
\hline
type & (a) & (b) & (c1) & (c2) & (d1) & (d2) \\ \hline
$\mu(R)$ & $(1)$ & $(2)$ & $(1,1)$ & $ (3)$ & $(1,1,1)$ & $(4)$ \\ \hline 
\end{tabular}
\end{center}
\vspace{.2cm}
Indeed, $\mu_j=\dim_{\CC} \left(\CC [x,y]\big/ (p_X, p_y) \right)$, where $p(x,y)$ is the local equation of $D$ at $q_j$ with $p_x$ and $p_y$ its partial derivatives. The number $\mu_j$ is called the {\it Milnor number} of $D$ at $q_j$; see \cite{Ri}. In particular, each type in the table is distinguished from the others by its Milnor numbers. 

\begin{proposition}\label{stabslope}
For any hyperplane $H\in (\PP^3)^*$ and $D=S \cap H$, the logarithmic tangent vector bundle $\Tt_S(-\log D)$ is $\mu$-stable.
\end{proposition}

\begin{proof}
We already proved the assertion for smooth $D$ in part (b) of the proof of Proposition \ref{stab}. For singular $D$, we follow the same strategy used in the smooth case, i.e. we assume the existence of a destabilizing line bundle $\Ll=\Oo_S(aL+\sum_{i=1}^6 b_iE_i)$. Then the same argument works in verbatim, until one can reach 
\[
\Ll{\cdot}\Hh=3a+\sum_{i=1}^6 b_i=0.
\]
\vspace{.3cm}
\noindent {\it{Claim 1} :}  \quad For each $1\le i < j \le 6$, we have $a+b_i+b_j \leq 1$. 

\noindent {\it Proof of Claim 1} : \quad 
Consider the line $L_{ij}\in|\Oo_{S}(L-E_i-E_j)|$ for each $1\leq i<j\leq 6$. First assume that $D$ is irreducible, which is the case of (a) and (b) in the above. We have $|(L_{ij}\cap D)_{\text{red}}|=1$ and so $\left(\Tt_{S}(-\log D)\right)_{|L_{ij}}\cong\Oo_{L_{ij}}(-1)\dirsum\Oo_{L_{ij}}(1)$ by \cite[Lemma 2.2]{HMPV}. On the other hand, we have $\Ll_{|L_{ij}}\cong\Oo_{L_{ij}}(a+b_i+b_j)$, we get the assertion. Now assume that $D$ is not irreducible, which is the case of (c1)$\sim$(d2). By choosing a different blow-down to $\PP^2$ and an order of indices of exceptional divisor, we may assume the following: \vspace{.2cm}
\begin{itemize}[itemindent=80pt, topsep=0pt, label=-, itemsep=0pt]
\item [$\bullet$ Cases (c1 \& c2)]: $C'\in|\Oo_{S}(2L-E_1-E_2-E_3-E_4)|,\quad L'\in|\Oo_{S}(L-E_5-E_6)|$\\
\item [$\bullet$ Cases (d1 \& d2)]: $L_1\in |\Oo_S(L-E_1-E_4)|, \quad L_2\in |\Oo_S(L-E_2-E_5)|, \quad L_3\in |\Oo_S(L-E_3-E_6)|$,\vspace{.2cm}
\end{itemize}
where we have either $D=C'+L'$ or $D=L_1+L_2+L_3$. For each case, if $L_{ij}$ is not a component of $D$, then we have $|L_{ij}\cap D|=1$ and $\left(\Tt_{S}(-\log D)\right)_{|L_{ij}}\cong\Oo_{L_{ij}}(-1)\dirsum\Oo_{L_{ij}}(1)$ by \cite[Lemma 2.2]{HMPV}. If $L_{ij}$ is a component of $D$, we obtain the following exact sequence by tensoring the Poincar\'e residue sequence with $\Oo_{L_{ij}}$, 
\[
0\to \mathcal{T}\!\mathrm{or}^1_{S}\left( \Oo_{L_{ij}}(-1), \Oo_{L_{ij}} \right) \cong \Oo_{L_{ij}}\to\left(\Tt_{S}(-\log D)\right)_{|L_{ij}}\to\left(\Tt_{S}\right)_{|L_{ij}}\stackrel{\phi}{\longrightarrow}\Oo_{L_{ij}}(-1)\dirsum\Oo_V\to 0,
\]
where $V$ is a $0$-dimensional subscheme of $S$, given as the intersection $L_{ij}\cap (D-L_{ij})$. Note that $|V|=2$ for any case of (c1)$\sim$(d2).  
Then, since $\left(\Tt_{S}\right)_{|L_{ij}}\cong\Oo_{L_{ij}}(-1)\dirsum\Oo_{L_{ij}}(2)$ we have $\ker(\phi)\cong\Oo_{L_{ij}}$ and thus $\left(\Tt_{L_{ij}}(-\log D)\right)_{|L_{ij}}\cong\Oo_{L_{ij}}\dirsum\Oo_{L_{ij}}$. In summary, 
\[
\left(\Tt_{S}(-\log D)\right)_{|L_{ij}}\cong
\left\{
\begin{array}{cc}
    \Oo_{L_{ij}}(-1)\dirsum\Oo_{L_{ij}}(1) &\text{if\;\;$L_{ij}\notin D$}\\
    \Oo_{L_{ij}}\dirsum\Oo_{L_{ij}} &\text{if\;\; $L_{ij}\in D$},
\end{array}
\right.
\]
and we get the assertion.
\qed

\vspace{.3cm}
\noindent As in part (b) of the proof of Proposition \ref{stab}, we can use {\it Claim 1} to conclude that $a+b_i+b_j=0$ for each $1\leq i<j\leq 6$. Note that we again obtain 
\[
\left(\Tt_{S}(-\log D)\right)_{|\widehat{L}_{i}}\cong\Oo_{\widehat{L}_{i}}(-1)\dirsum\Oo_{\widehat{L}_{i}}(1)
\]
for any $D$ by \cite[Lemma 2.2]{HMPV}, and so one can obtain $\Ll \cong \Oo_S(-2L+\sum_{i=1}^6 E_i)$ using the exact same argument in part (b) of the proof of Proposition \ref{stab}. Setting $\Ll\otimes \Oo_S(D')$ the saturation of $\Ll$ in $\Tt_S$, one can obtain the diagram (\ref{stabdiag}) with $\Oo_D(D)$ replaced by $\Jj_D(D)$. On the other hand, we consider a commutative diagram
\begin{equation}\label{stabdiag2}
\begin{tikzcd}[
  row sep=normal, column sep=normal,
  ar symbol/.style = {draw=none,"\textstyle#1" description,sloped},
  isomorphic/.style = {ar symbol={\cong}},
  ]
  &        & 0 \ar[d] & 0\ar[d]\\
  0 \ar[r] & \Ll\tensor\Oo_{D'}(D') \ar[d,isomorphic]\ar[r] & \Jj_{D}(D) \ar[d]\ar[r] & \Kk \ar[d] \ar[r] & 0  \\
  0 \ar[r] & \Ll\tensor\Oo_{D'}(D') \ar[r]             & \Oo_{D}(D)   \ar[d]\ar[r] & \Kk' \ar[d] \ar[r] & 0\\
           &                                    & \Oo_{D^s} \ar[d]\ar[r,isomorphic] & \Oo_{D^s} \ar[d] & \\
    &  & 0 & 0.
\end{tikzcd}
\end{equation}
Here, $\Jj_{D}:=\mathrm{J}(D){\cdot}\Oo_{D}$, where $\mathrm{J}(D)$ is the \textit{Jacobian ideal sheaf} of $D$ in $\Oo_{S}$. If we set $D^s$ for the closed subscheme of $D$ defined by $\Jj_{D}$ so that we have an exact sequence
\[
0\to\Jj_{D}\to\Oo_{D}\to\Oo_{D^s}\to 0,
\]
then the sheaf $\Oo_{D^s}$ is supported on the singular locus of $D$; see \cite[Page 36]{dolgachev2007logarithmic}.

\vspace{.3cm}
\noindent {\it{Claim 2} :}  \quad We have $D'=0$, i.e. $\Ll$ is saturated in $\Tt_S$. 

\noindent {\it Proof of Claim 2} : \quad Suppose that $D$ is irreducible. If $D'$ is nontrivial, then from the second horizontal sequence in \eqref{stabdiag2} it would be $D'=D$. Notice that $\Ll\tensor\Oo_{D'}(D')\cong\Oo_{D}(L)$, while $D.(D-L)=D.D-D.L=0$, which implies $\Kk'=0$, absurd. Now assume that $D$ is not irreducible. If $D'$ is nontrivial, then we have the following three possibilities: 
\[
\begin{array}{ccc}
 \text{(i)}\; D'=D &\quad \text{(ii)}\; D'=C' \;\;\text{or}\;\;L_1+L_2   &\quad \text{(iii)}\; D'=L' \;\;\text{or}\;\;L_1 
\end{array}
\]
up to reordering of indices. For the case (i), we have $\Ll\tensor\Oo_{S}(D')\cong\Oo_{S}(L)$ and this is impossible as treated in the above. For (ii), assuming $D'=C'$ we would have $\Ll\tensor\Oo_{S}(D')\cong\Oo_{S}(-E_5-E_6)$ and then $\mu(\Oo_{S}(-E_5-E_6))=2>\frac{3}{2}=\mu(\Tt_{S})$, contradicting the stability of $\Tt_S$. Now for (iii) we can set $D'=L'$; the case $D'=L_1$ can be dealt verbatim. It implies that $\Ll\tensor\Oo_{S}(D')\cong\Oo_{S}(-L+E_1+E_2+E_3+E_4)$ and that there exists an extension
\begin{equation}
    0\to\Oo_{S}(-L+\sum_{i=1}^4 E_i)\to\Tt_{S}\to\Ii_{Z'',S}\tensor\Oo_{S}(4L-\sum_{i=1}^4 2E_i-E_5-E_6)\to 0
\end{equation}
with $|Z''|=5$. When we restrict the sequence to the line $L_{ij}$ with $i\in\{1,2,3,4\}$ and $j\in\{5,6\}$, we obtain
\begin{equation}\label{extt}
0\to\Oo_{L_{ij}}(e)\to\left(\Tt_{S}\right)_{|L_{ij}}\cong\Oo_{L_{ij}}(-1)\dirsum\Oo_{L_{ij}}(2)\to\Oo_{L_{ij}}(1-e)\to 0
\end{equation}
for $e=|Z''\cap L_{ij}|$. As an automatic consequence we can obtain $e=2$, \ie each line $L_{ij}$ with $i\in\{1,2,3,4\}$ and $j\in\{5,6\}$ passes through exactly two points of $Z''$. But this is impossible by the following argument. Let $Z''=\{q_1,\ldots, q_5\}$ and suppose that there exists such an arrangement of $L_{ij}$'s. First, we may assume that $q_1$ and $q_2$ is contained in $L_{15}$. Second, $|L_{16}\cap L_{15}|=0$ implies that $L_{16}$ cannot contains $q_1$ or $q_2$. Suppose $q_3, q_4\in L_{16}$. Third, for the line $L_{25}$, we have $|L_{25}\cap L_{15}|=0$ and $|L_{25}\cap L_{16}|=1$. It implies that $L_{25}$ could contain at most one of $q_3, q_4$ with $q_5$. Next, for $L_{26}$ we must have that $|L_{26}\cap L_{15}|=0$, $|L_{26}\cap L_{16}|=0$, and $|L_{26}\cap L_{25}|=0$. In this circumstances, $L_{26}$ could contain at most one of $q_1, q_2$ but not any of $q_3,q_4$, or $q_5$. Therefore, the arrangement that we supposed could not exist. \qed

\vspace{.3cm}
\noindent Now by Claim 2, one can obtain an exact sequence
\[
0\to \Ll \to \Tt_S \to \Ii_{Z'', S}\otimes \Gg \to 0,
\]
where $Z''$ is a $0$-dimensional subscheme of $S$ with $|Z''|=7$. Notice that we have reached the same extension $\eqref{ext2}$ in part (b) of the proof of Proposition \ref{stab}. We can proceed by the same argument to obtain a contradiction. Consequently, $\Tt_{S}(-\log D)$ is $\mu$-stable with respect to $\Oo_{S}(1)$, for any hyperplane $H\in(\PP^3)^{*}$ and $D=S\cap H$.   \end{proof}

\begin{proposition}\label{stabstab}
For each hyperplane $H\in (\PP^3)^*$, the net logarithmic tangent sheaf $\Tt_S(H; \PP^3)$ is stable. 
\end{proposition}

\begin{proof}
Assume that $\Tt_S(H; \PP^3)$ is not stable with a destabilizing subsheaf $\Ff$. By saturating $\Ff$, we can set $\Ff\cong \Ii_{Z_1, S}\otimes \Ll$ for some line bundle $\Ll$ and a $0$-dimensional subscheme $Z_1$. Then it admits an exact sequence
\[
0\to \Ii_{Z_1, S} \otimes \Ll \to \Tt_S(H; \PP^3) \to \Ii_{Z_2, S}\otimes \Ll^{-1} \to 0,
\]
for another $0$-dimensional subscheme $Z_2$ with $|Z_1|+|Z_2|-\Ll{\cdot}\Ll = 9$. Dualizing this sequence, one can get an injection
\[
0\to \Ll \cong \mathcal{H}om_S(\Ii_{Z_2, S}\otimes \Ll^{-1}, \Oo_S) \to \Tt_S(H; \PP^3)^\vee \cong \Tt_S(-\log D)^\vee \cong \Tt_S(-\log D)
\]
with $D=S \cap H$. By Proposition \ref{stabslope}, one gets $\Ll {\cdot}\Hh\leq -1$ and so 
\[
\mathrm{P}_{\Ff}(t)=\frac{3}{2}t^2+\left( \frac{3}{2}+\Ll{\cdot}\Hh \right) t+\left ( 1+\frac{1}{2}\left(\Ll{\cdot}\Hh+\Ll {\cdot}\Ll \right) -|Z_1|\right),
\]
which is strictly less than $\mathrm{P}_{\Ee}(t)=\frac{3}{2}t(t+1)-\frac{7}{2}$ with $\Ee=\Tt_S(H; \PP^3)$ for $t\gg 0$, a contradiction. 
\end{proof}

\begin{corollary}\label{cinj2}
The map $\Phi :  \PP\mathrm{H}^0(\Oo_S(-K_S))\cong\PP^3\dashrightarrow\mathbf{M}_S^{\Hh}(0,9)$ in Corollary \ref{cinj1}, is an injective morphism. 
\end{corollary}

\begin{proof}
Using the argument in the proof of Corollary \ref{corr}, one can construct the universal sheaf $\Pp$ on $S\times (\PP^3)^*$ such that 
\[
\Pp_{~|S\times \{H\}} \cong \Tt_S(H;\PP^3)
\]
for each $H\in (\PP^3)^*$, which induces a morphism $\Phi : (\PP^3)^* \rightarrow \mathbf{M}_S^{\Hh}(0,9)$ by Proposition \ref{stabstab}. Thus it remains to show the injectivity of $\Phi$. Now, from Corollary \ref{cinj1} and the fact that $\mathrm{Sing}(\Tt_S(H; \PP^3))=\mathrm{Sing}(S \cap H)$, one can deal with the case when $S \cap H$ is singular. Set $\Ee:=\Tt_S(H; \PP^3)$ for some $H$ with $D=S \cap H$ is singular. In particular, we have $\mathrm{Sing}(\Ee)\ne \emptyset$ and so pick a point $p\in \mathrm{Sing}(\Ee)$. Then it is enough to check that $H=T_pS$ the tangent plane of $S$ at $p$, which would define the inverse morphism $\Phi^{-1} : \mathrm{Im}(\Phi) \rightarrow \PP^3$. Choose the coordinates $[x_0:x_1:x_2:x_3]$ of $\PP^3$ with $H=\VV(x_3)$. Then we can set $S=\VV (F)$ with
\[
F(x_0, x_1, x_2, x_3)=A(x_0,x_1,x_2)+x_3\cdot B(x_0,x_1,x_2,x_3).
\]
Recall that $\mathrm{Sing}(D)$ is the degeneracy locus of the Jacobian matrix 
\[
\begin{pmatrix}
    \frac{\partial A}{\partial x_0}+x_3{\cdot}\frac{\partial B}{\partial x_0}&\frac{\partial A}{\partial x_1}+x_3{\cdot}\frac{\partial B}{\partial x_1}&\frac{\partial A}{\partial x_2}+x_3{\cdot}\frac{\partial B}{\partial x_2}&B+x_3{\cdot}\frac{\partial B}{\partial x_3}\\
    0&0&0&1
\end{pmatrix}
\]
in $D$, given by the $(2\times 2)$-minors. In particular, the point $p=[p_0:p_1:p_2:0]\in D$ satisfies the equations
\[
\left(\frac{\partial A}{\partial x_i}+x_3{\cdot}\frac{\partial B}{\partial x_i}\right)(p)=0 \quad\text{for}\quad i=0,1,2.   
\]
Note also that we have $B(p)\neq 0$, since $S$ is smooth. Finally, the tangent plane $T_{p}S$ of $S$ at $p$ is given by 
\[
\sum_{i=0}^2\left(\left(\frac{\partial A}{\partial x_i}+x_3{\cdot}\frac{\partial B}{\partial x_i}\right)(p){\cdot}(x_i-p_i)\right)+\left(B+x_3{\cdot}\frac{\partial B}{\partial x_3}\right)(p)\cdot x_3=0,
\]
which is $x_3=0$.
\end{proof}

\begin{remark}
Let $[\Ee]\in \mathbf{M}_S^{\Hh}(0,9)$ be a stable vector bundle. Note that 
\[
\mathrm{ext}_S^1(\Ee, \Ee)_0-\mathrm{ext}_S^2(\Ee, \Ee)_0=4{\cdot} c_2(\Ee)-3\chi(\Oo_S)=33,
\]
and so the expected dimension of $\mathbf{M}_S^{\Hh}(0,9)=33$. From $\chi(\Ee(2))=11$ and $\mathrm{h}^2(\Ee(2))=\mathrm{h}^0(\Ee(-3))$ by the Serre duality, we get $\mathrm{h}^0(\Ee(2))\ge 11$. A general section $\sigma \in \mathrm{H}^0(\Ee(2))$ induces an exact sequence
\begin{equation}\label{ext944}
0\to \Oo_S \stackrel{\sigma}{\longrightarrow} \Ee(2) \to \Ii_{Z', S}(4) \to 0,
\end{equation}
for a $0$-dimensional subscheme $Z'$ with $|Z'|=21$. Conversely, for a general choice of $0$-dimensional subscheme $Z'$ with $|Z'|=21$, the extension family (\ref{ext944}) is parametrized by $\mathrm{Ext}_S^1(\Ii_{Z', S}(4), \Oo_S) \cong \mathrm{H}^1(\Ii_{Z', S}(3))^\vee\cong \CC^{2}$. Since $\mathrm{h}^0(\Ee(2))=\mathrm{h}^0(\Ii_{Z', S}(4))+\mathrm{h}^0(\Oo_S)=11$, we get that the vector bundles fitting into the extension (\ref{ext944}) form a family whose dimension is $2\times 21+2-11=33$. In particular, $\mathbf{M}_S^{\Hh}(0,9)$ is generically smooth with dimension $33$, and we have
\[
\mathrm{h}^0(\Ee(1))=0, \quad \mathrm{h}^1(\Ee(1))=1, \quad \mathrm{h}^0(\Ee(2))=11, \quad  \mathrm{h}^1(\Ee(2))=0
\]
for a general vector bundle $[\Ee]\in \mathbf{M}_S^{\Hh}(0,9)$. 
\end{remark}

\begin{proposition}\label{char}
The logarithmic vector bundles $\Tt_S(-\log S \cap H)$ in $\mathbf{M}_S^{\Hh}(0,9)$ are characterized as the vector bundles $[\Ff]\in \mathbf{M}_S^{\Hh}(0,9)$ satisfying the condition that $\Ff(1)$ is globally generated with $\mathrm{h}^0(\Ff(1))=3$. 
\end{proposition}

\begin{proof}
We have an exact sequence
\begin{equation}\label{ext99934}
0\to \Oo_S(-3) \stackrel{\nabla}{\longrightarrow} \Oo_S(-1)^{\oplus 3} \to \Ff \to 0
\end{equation}
for $\Ff:=\Tt_S(H; \PP^3)\cong \Tt_S(-\log D)$ with $D=S \cap H$ smooth, from which we get that $\Ff(1)$ is globally generated with $\mathrm{h}^0(\Ff(1))=3$. Conversely, let $[\Ff]\in \mathbf{M}_S^{\Hh}(0,9)$ be a stable vector bundle, where $\Ff(1)$ is globally generated and $\mathrm{h}^0(\Ff(1))=3$ so that it fits into the extension (\ref{ext99934}). Noticing $\mathrm{H}^0(\Oo_S(2))\cong \mathrm{H}^0(\Oo_{\PP^3}(2))\cong \CC^{\oplus 10}$, the map $\nabla$ is defined by $(Q_0, Q_1, Q_2)$ a triple of quadratic forms in $\CC[x_0, x_1, x_2, x_3]$. Consider a map 
\[
\partial ~:~ \mathrm{H}^0(\Oo_{\PP^3}(3)) \times \CC \langle\partial_0, \partial_1, \partial_2, \partial_3\rangle  \longrightarrow \mathrm{H}^0(\Oo_{\PP^3}(2))
\]
defined by derivation, where $\partial_j:=\frac{\partial}{\partial x_j}$ and $\CC \langle\partial_0, \partial_1, \partial_2, \partial_3\rangle $ is the linear span of them. Since the map $\partial$ is surjective, $\mathrm{ker}(\partial)$ is $14$-dimensional. Moreover, if $p_1 : \mathrm{ker}(\partial) \rightarrow \mathrm{H}^0(\Oo_{\PP^3}(3))$ is the projection to the first factor, then $\mathrm{Im}(p_1)$ is also $14$-dimensional. In other words, when we set
\[
V_i:=\left\{ F \in \mathrm{H}^0(\Oo_{\PP^3}(3))~\bigg|~ \sum_{j=0}^3 c_j \frac{\partial F}{\partial x_j}=Q_i \text{ for some }(c_0, c_1, c_2, c_3)\in \CC^{\oplus 4}\right\}\cup \{0\}
\]
for each $i=0,1,2$, then we get $\dim V_i=14$ and so $V_0\cap V_1 \cap V_2 \not\ne (0)$. In particular, up to change of coordinates, we can write $Q_i=\frac{\partial F}{\partial x_i}$ for some $F\in \mathrm{H}^0(\Oo_{\PP^3}(3))$, concluding the proof.  
\end{proof}

\end{document}